\documentclass[preprint,1p,10pt]{elsarticle}

\usepackage{lineno,hyperref}
\modulolinenumbers[1]

\usepackage{array}
\newcolumntype{C}[1]{>{\centering\arraybackslash}m{#1}}

\usepackage{graphics}
\usepackage{amssymb, latexsym, amsmath, epsfig}
\usepackage{graphicx}
\usepackage{color}
\usepackage{epstopdf}
\usepackage{soul}
\usepackage[normalem]{ulem}
\usepackage{float}
\usepackage{amsfonts,amstext,amsthm}
\usepackage{epsfig}
\usepackage{subfigure}

\graphicspath{{figures/}}

\newtheorem{theorem}{Theorem}
\newtheorem{lemma}{Lemma}

\newtheorem{coro}{Corollary}
\numberwithin{equation}{section}
\newtheorem{remark}{Remark}
\newcommand{\bfs}[1]{{\boldsymbol #1}}

\biboptions{numbers,sort&compress}

\journal{arXiv}

\begin{document}

\begin{frontmatter}

\title{Dispersion optimized quadratures for isogeometric analysis}

\author[ad,ad1,ad2]{Victor Calo}
\ead{Victor.Calo@curtin.edu.au}

\author[ad,ad1]{Quanling Deng\corref{corr}}
\cortext[corr]{Corresponding author}
\ead{Quanling.Deng@curtin.edu.au}

\author[ad,ad1]{Vladimir Puzyrev}
\ead{Vladimir.Puzyrev@curtin.edu.au}

\address[ad]{Applied Geology, Curtin University, Kent Street, Bentley, Perth, WA 6102, Australia}
\address[ad1]{Curtin Institute for Computation, Curtin University, Kent Street, Bentley, Perth, WA 6102, Australia}
\address[ad2]{Mineral Resources, Commonwealth Scientific and Industrial Research Organisation (CSIRO), Kensington, Perth, WA 6152, Australia}

\begin{abstract}
We develop and analyze quadrature blending schemes that minimize the dispersion error of isogeometric analysis  up to polynomial order seven with maximum continuity in the span. The  schemes yield two extra orders of convergence (superconvergence) on the eigenvalue errors, while the eigenfunction errors are of optimal convergence order. Both dispersion and spectrum analysis are unified in the form of a Taylor expansion for eigenvalue errors. The resulting schemes increase the accuracy and robustness of isogeometric analysis for wave propagation as well as the differential eigenvalue problems. We also derive an a posteriori error estimator for the eigenvalue error based on the superconvergence result. We verify with numerical examples the analysis of the performance of the proposed schemes.

\end{abstract}

\begin{keyword}
isogeometric analysis \sep quadrature rule \sep dispersion analysis \sep spectrum analysis
\end{keyword}

\end{frontmatter}


\section{Introduction} \label{sec:intr} 
Isogeometric analysis is a widely-used numerical method introduced by Hughes and his collaborators \cite{hughes2005isogeometric,cottrell2009isogeometric} in 2005. Approximation, stability, and error estimates are established in \cite{bazilevs2006isogeometric}.  Structural vibrations and wave propagation problems are investigated using isogeometric analysis in \cite{cottrell2006isogeometric}. Spectrum analysis of the method shows that the method significantly improves accuracy in the spectral calculation over the classical finite element method \cite{cottrell2009isogeometric}. Further advantages of the method on spectral approximation properties are investigated in \cite{hughes2014finite}.

A duality principle, which induces a bijective map from spectral analysis to dispersion analysis, is established \cite{hughes2008duality}. The spectral analysis for structural vibrations (eigenvalue problems) and the dispersion analysis for wave propagation are then unified. 
Although the cost per degree of freedom of isogeometric analysis is higher than for finite elements \cite{collier2012cost,collier2013cost,collier2014computational,pardo2012survey}, the dispersion error is dramatically smaller than that of finite elements.
In this paper, we propose a blending scheme for isogeometric analysis which increases the convergence by two orders with respect to the mesh size. With this motivation, under the framework of unified analysis, we study the dispersion optimization of the isogeometric analysis.

Dispersion analysis for the Galerkin finite element methods has been studied actively in literature; see for example Thomson and Pinsky \cite{thompson1994complex,thompson1995galerkin}, Ihlenburg and Babuska \cite{ihlenburg1995dispersion}, Ainsworth \cite{ainsworth2004discrete}, and others \cite{harari1997reducing,harari2000analytical,he2011dispersion}. In \cite{thompson1994complex}, Thomson and Pinsky study the dispersive effects of the Galerkin methods with different local approximation basis (Legendre, spectral, and Fourier elements) for the Helmholtz equation and it is found that the choice of the basis has a negligible effect on the dispersion  errors. In \cite{ainsworth2004discrete}, a $2p$ convergence rate of the dispersion error is shown for arbitrary $p$-th order finite element methods. For the linear case, a fourth order superconvergence result is obtained by a modified integration rule for finite elements in \cite{guddati2004modified}.

As early as 1984, Marfurt \cite{marfurt1984accuracy} conjectured that the most promising and cost-effective method for computational wave propagation is to employ a weighted average of the finite element and spectral element schemes. In 2010, Ainsworth and Wajid \cite{ainsworth2010optimally}  introduced the optimal blending of these two schemes of arbitrary polynomial order. This optimal blending delivers two extra orders of convergence on the dispersion errors. 
In 2016, a dispersion optimized edge-based mimetic finite difference method for Maxwell's equations in cold plasma was developed in \cite{bokil2016dispersion}. The authors presented a generalized form of mass lumping and an optimization procedure to reduce the numerical dispersion error from second to fourth order accuracy.  In \cite{wajid2012optimally}, the authors described a similar result on dispersion by optimally blending finite element and spectral element methods for Maxwell's equations.

To evaluate the bilinear form, in the case of spectral and finite elements, Ainsworth and Wajid pointed out  that spectral elements use the Gauss-Lobatto quadrature rule while finite elements use the Gauss-Legendre quadrature rule \cite{ainsworth2010optimally} on the same polynomial space. The optimal blending of spectral and finite elements is essentially a blending of the Gauss-Legendre and Gauss-Lobatto quadrature rules. Thus, this new blended quadrature rule minimizes the dispersion errors. 

In this paper, we minimize the dispersion error for isogeometric analysis by blending quadrature rules appropriately while the eigenfunction errors do not degenerate. We study optimally-blended quadratures for isogeometric analysis up to polynomial order seven. We minimize the dispersion error  to obtain two additional orders of error convergence.

The rest of this paper is organized as follows. Section \ref{sec:ps} presents the discretization of an eigenvalue problem and its generalized Pythagorean eigenvalue theorem. In Section \ref{sec:ddr}, we present the discrete dispersion relation and derive the optimized dispersion error expansions for isogeometric analysis up to order seven. Dispersion and spectrum analysis are unified in the form of a Taylor expansion for the eigenvalue errors.  Section \ref{sec:ea} analyzes the error of the blending schemes and an a posteriori error estimator based on the superconvergence result of the eigenvalue error is derived. Section \ref{sec:num} presents numerical examples to demonstrate the performance of the proposed blending schemes. Section 6 describes our concluding remarks.

\section{Problem setting} \label{sec:ps}
We consider stationary waves as described by the Helmholtz equation
\begin{equation} \label{eq:pde}
\begin{aligned}
\Delta u +  \omega^2 u & = 0 \quad  \text{in} \quad \Omega, \\
u & = 0 \quad \text{on} \quad \partial \Omega,
\end{aligned}
\end{equation}
where $\Omega = (0, 1)^d, d \in \{1,2,3\}, $ is a unit cube (unit interval when $d=1$ and unit square when $d=2$), $\Delta = \nabla^2$ is the Laplacian and $\omega = \omega_f/c$ with $\omega_f$ being the frequency of a particular sinusoidal wave and $c$ being the speed of sound of the medium. Denoting $\lambda = \omega^2$, in the view of duality and unified analysis of discrete approximations for wave propagation studied in \cite{hughes2008duality}, \eqref{eq:pde} is also posed as a second order elliptic eigenvalue problem: find real values $\lambda $ and nonzero functions $u$ such that 
\begin{equation} \label{eq:pdee}
-\Delta u =  \lambda u \quad  \text{in} \quad \Omega.
\end{equation}

The eigenvalue problem \eqref{eq:pdee} has a countable set of eigenvalues $\lambda_j \in \mathbb{R}^+$ (c.f., \cite[Sec.~9.8]{brezis2010functional} and \cite{gilbarg2015elliptic,babuvska1991eigenvalue,boffi2010finite,strang1973analysis})
\begin{equation}
0 < \lambda_1 < \lambda_2 \leq \lambda_3 \leq \cdots
\end{equation}
and an associated set of orthonormal eigenfunctions $u_j$
\begin{equation} \label{eq:uon}
(u_j, u_k) = \delta_{jk},
\end{equation}
where $(\cdot, \cdot)$ denotes the $L^2$-inner product on $\Omega$. Herein, the Kronecker delta is defined as $\delta_{lm} =1$ when $l=m$ while zero otherwise.

\subsection{Discretization}
For an open bounded set $S \subset \mathbb{R}^d$ with Lipschitz boundary, we denote by $H^m(S)$  the Sobolev spaces and $H^m_0(S)$ the Sobolev spaces with functions vanishing at the boundary, where $m>0$ specifies the order of the weak derivative.
The variational formulation of \eqref{eq:pdee} is to find $\lambda \in \mathbb{R}^{+}$ and $u \in H^1_0(\Omega)$ such that 
\begin{equation} \label{eq:vf}
a(u, v) =  \lambda b(u, v) \quad \forall \ v \in H^1_0(\Omega), 
\end{equation}
where 
$
a(w, v) = (\nabla w, \nabla v)
$ and $
b(w, v) = (w, v)
$. Let $(\lambda_j, u_j)$ be an eigenpair, then $a(u_j, v) =  \lambda_j b(u_j, v)$. For orthonormal eigenfunctions $u_j$ (in the sense of $L^2$-inner product; see \eqref{eq:uon}), they are also orthogonal in the energy inner product
\begin{equation} \label{eq:vfo}
a(u_j, u_k) =  \lambda_j b(u_j, u_k) = \lambda_j \delta_{jk}.
\end{equation}

Let $\prod$ be the product symbol and $N_k$, $k = 1, \cdots, d,$ be positive integers associated with the space variable $x_k$.
Assume that a uniform tensor product mesh with $\prod_{k=1}^d N_k $ elements is placed on $\overline\Omega=[0,1]^d$ with grid nodes located at $(h_1 n_1, \cdots, h_d n_d)$, where $h_k = \frac{1}{N_k}, k=1, \cdots, d,$ is the size of the $k^\text{th}$ dimension and $n_k = 0, 1, \cdots, N_k.$ In case of one dimension, we simplify the notation as $N, h,$ respectively. We denote each element as $K$ and their collection as $\mathcal{T}_h$ such that $\bar\Omega = \cup_{K\in \mathcal{T}_h}  K$. Due to the tensor product structure of the discretization, the element size is $h = \sqrt{ \sum_{k=1}^d h_k^2}$.
The Galerkin-type numerical methods seek $\lambda^h \in \mathbb{R}^+$ and $u^h \in V_h \subset H_0^1(\Omega)$ such that 
\begin{equation} \label{eq:vfh}
a(u^h, v_h) =  \lambda^h b(u^h, v_h) \quad \forall \ v_h \in V_h. 
\end{equation}

Different solution/trial spaces $V_h$ lead to different numerical methods. Under this framework, we consider the following methods: finite element, spectral element, and isogeometric analysis. We utilize the B-spline basis functions for isogeometric analysis. 
Following \cite{de1978practical,cottrell2009isogeometric,piegl2012nurbs}, the definition of the $p$-th order B-spline basis functions in one dimension is as follows. 
Let $\Xi = \{\xi_0, \xi_1, \cdots, \xi_{N_0} \}$ be an ordered knot vector with $0=\xi_0 \le \xi_1 \le \cdots \le \xi_{N_0} = 1$, that is, a nondecreasing sequence of real numbers called knots.  The B-spline basis function of degree $p$, denoted as $B_a^p(\xi)$ (here $a$ refers to an index with slightly abuse of notation $a(\cdot, \cdot)$), is defined as 
\begin{equation} \label{eq:B-spline}
\begin{aligned}
B_a^0(\xi) & = 
\begin{cases}
1, \quad \text{if} \ \xi_a \le \xi < \xi_{a+1}, \\
0, \quad \text{otherwise}, \\
\end{cases} \\ 
B_a^p(\xi) & = \frac{\xi - \xi_a}{\xi_{a+p} - \xi_a} B_a^{p-1}(\xi) + \frac{\xi_{a+p+1} - \xi}{\xi_{a+p+1} - \xi_{a+1}} B_{a+1}^{p-1}(\xi).
\end{aligned}
\end{equation}

In this paper, for isogeometric analysis, we utilize B-splines on uniform tensor product meshes with non-repeating knots, that is, the B-splines with maximum continuity, while for finite element method, we utilize $C^0$ B-spline basis functions. For multiple dimensions, the B-spline basis functions are constructed by tensor products of these one-dimensional B-spline basis functions; we refer to \cite{cottrell2009isogeometric,piegl2012nurbs} for details.
Let $B^p_{n_k}$ be the one-dimensional basis functions in dimension $k=1,\cdots, d$ for $p$-th order finite element and isogeometric analysis. Provided the tensor product structure of the discretization, a $d$-dimensional basis function can be written as $\prod_{k=1}^d B^p_{n_k}$. 
Then $ V_h = \text{span} \{  \prod_{k=1}^d B^p_{n_k} \}_{ n_k = 0,1, \cdots, N_k}$.

In the framework of finite elements, the eigenpairs $(\lambda^h, u^h)$ have the following properties (see for example, \cite{babuvska1991eigenvalue,strang1973analysis, banerjee1989estimation, banerjee1992note})
\begin{equation} \label{eq:femeverr}
| \lambda - \lambda^h | \le C h^{2p} \| u \|^2_{p+1, \Omega} \qquad \text{and} \qquad \| u - u^h \|_{1, \Omega} \le C h^p\| u \|_{p+1, \Omega},
\end{equation}
where $C$ is a constant independent of $h$.

In practice, the integrals involved in $a(u^h, v_h) $ and $b(u^h, v_h)$ are evaluated numerically, that is, approximated by quadrature rules. On a reference element $\hat K$, a quadrature rule is of the form
\begin{equation} \label{eq:qr}
\int_{\hat K} \hat f(\hat{\boldsymbol{x}}) \ \text{d} \hat{\boldsymbol{x}} \approx \sum_{l=1}^{N_q} \hat{\varpi}_l \hat f (\hat{n_l}),
\end{equation}
where $\hat{\varpi}_l$ are the weights, $\hat{n_l}$ are the nodes, and $N_q$ is the number of quadrature points. For each element $K$, we assume that there is an invertible affine map $\sigma$ such that $K = \sigma(\hat K)$, which leads to the correspondence between the functions on $K$ and $\hat K$. Assuming $J_K$ is the corresponding Jacobian of the mapping, \eqref{eq:qr} induces a quadrature rule, denoted as $\mathcal{Q}$, over the element $K$ given by
\begin{equation} \label{eq:q}
\int_{K}  f(\boldsymbol{x}) \ \text{d} \boldsymbol{x} \approx \mathcal{Q}(f) = \sum_{l=1}^{N_q} \varpi_{l,K} f (n_{l,K}),
\end{equation}
where $\varpi_{l,K} = \text{det}(J_K) \hat \varpi_l$ and $n_{l,K} = \sigma(\hat n_l)$.
Let $G_{N_q}$ and $GL_{N_q}$ denote the $N_q$-point Gauss-Legendre quadrature rule and Gauss-Lobatto quadrature rule, respectively. The detailed description of these rules is given in \cite{kythe2004handbook}. We define the blending quadrature rule $\mathcal{Q}_\tau$ as
\begin{equation} \label{eq:bq}
\int_{K}  f(\boldsymbol{x}) \ \text{d} \boldsymbol{x} \approx \mathcal{Q}_\tau (f) = \tau \mathcal{Q}_1 (f) + (1 - \tau ) \mathcal{Q}_2 (f),
\end{equation}
where $\tau$ is the blending parameter and $\mathcal{Q}_1, \mathcal{Q}_2$ are different quadrature rules. In this paper, we seek the optimal blending parameter to minimize the dispersion errors of the isogeometric analysis.
We denote by $O_p$ the optimal blending scheme for the $p$-th order isogeometric analysis.

Applying quadrature rules to \eqref{eq:vfh}, we have the approximated form
\begin{equation} \label{eq:vfho}
\tilde a_h(\tilde u^h, v_h) =  \tilde\lambda^h \tilde b_h(\tilde u^h, v_h) \quad \forall \ v_h \in V_h,
\end{equation}
where 
\begin{equation} \label{eq:ba}
\tilde a_h(w, v) = \sum_{K \in \mathcal{T}_h} \sum_{l=1}^{N_q} \varpi_{l,K}^{(1)} \nabla w (n_{l,K}^{(1)} ) \cdot \nabla v (n_{l,K}^{(1)} )
\end{equation}
and
\begin{equation} \label{eq:bb}
\tilde b_h(w, v) = \sum_{K \in \mathcal{T}_h} \sum_{l=1}^{N_q} \varpi_{l,K}^{(2)} w (n_{l,K}^{(2)} ) v (n_{l,K}^{(2)} ),
\end{equation}
where $\{\varpi_{l,K}^{(1)}, n_{l,K}^{(1)} \}$ and $\{\varpi_{l,K}^{(2)}, n_{l,K}^{(2)} \}$ specify two (possibly different) quadrature rules. 

We represent the eigenfunctions as a linear combination of the B-spline (or Lagrange) basis functions and substitute all the basis functions for $V_h$ in \eqref{eq:vfho}. Applying quadrature rules, this leads to the generalized matrix eigenvalue problem
\begin{equation} \label{eq:mevp}
\mathbf{K} \tilde{\mathbf{U}} = \tilde \lambda^h \mathbf{M} \tilde{\mathbf{U}},
\end{equation}
where $\mathbf{K}_{ab} = \tilde a_h(B_a, B_b), \mathbf{M}_{ab} = \tilde b_h(B_a, B_b)$ with $B_a$ and $B_b$ being generic basis functions, and $\tilde{\mathbf{U}}$ is the corresponding representation of the eigenvector as the coefficients of the basis functions. 
Once the matrix eigenvalue problem \eqref{eq:mevp} is solved, we obtain eigenpairs.
Throughout this paper, we refer to $(\lambda_j, u_j)$ as one exact eigenpair, $(\lambda_j^h, u_j^h)$ as one approximate eigenpair when the inner products are not modified by  the numerical quadrature, and $(\tilde \lambda_j^h, \tilde u_j^h)$ one approximate eigenpair when modified by the numerical quadrature.

\subsection{Pythagorean eigenvalue theorem and its generalization} \label{sec:petg}
Denoting the energy norm as $\| \cdot \|_{E} = \sqrt{a(\cdot, \cdot)}$, the Pythagorean eigenvalue theorem \cite{strang1973analysis,hughes2014finite} states 
\begin{lemma}\label{lem:pet}
For each discrete mode, with the normalization $\| u_j \|_{0, \Omega} = 1$ and $ \| u_j^h \|_{0, \Omega} = 1$, there holds
\begin{equation} \label{eq:pet0}
\| u_j - u_j^h \|_E^2 = \lambda_j \| u_j - u_j^h \|^2_{0, \Omega} + \lambda_j^h - \lambda_j.
\end{equation}
\end{lemma}

In general, applying quadrature rules to the inner products $a(\cdot, \cdot)$ and $b(\cdot, \cdot)$ results in quadrature errors. We denote by $\| \cdot \|_{E,h} = \sqrt{\tilde a_h(\cdot, \cdot)}$ the (approximate) energy norm evaluated by a quadrature rule \eqref{eq:ba}, the generalized Pythagorean eigenvalue theorem, c.f., \cite{puzyrev2017dispersion}, is stated as follows
\begin{theorem} \label{thm:gpet}
For each discrete mode, with the normalization $\| u_j \|_{0, \Omega} = 1$ and $\tilde b_h( \tilde u_j^h, \tilde u_j^h) = 1$, there holds
\begin{equation} \label{eq:gpet}
\| u_j - \tilde u_j^h \|_E^2 =  \tilde \lambda_j^h - \lambda_j + \lambda_j \| u_j - \tilde u_j^h \|^2_{0, \Omega} + \| \tilde u^h_j \|_E^2 - \| \tilde u^h_j  \|_{E,h}^2 + \lambda_j \Big( 1 - \| \tilde u^h_j  \|^2_{0, \Omega} \Big).
\end{equation}

\end{theorem}

\begin{remark} \label{rem:pt}
The Pythagorean eigenvalue theorem reveals the properties of the numerical approximation of the eigenvalue problem  when the inner products are fully represented in the discrete setting, while the generalized theorem  accounts for the effect of the numerical integration.
\end{remark}

\section{The discrete dispersion relation and dispersion optimization} \label{sec:ddr}

In the linear element case, isogeometric analysis and finite elements coincide and result in the same discrete dispersion relation; see for example \cite{ainsworth2004discrete, ainsworth2009dispersive, ainsworth2010optimally}.  We start with quadratic elements and focus on the one-dimensional case with $\Omega = [0, 1]$ to illustrate the idea. 
We seek an approximate solution of the form 
\begin{equation} \label{eq:1dsol}
U (x) = \sum_{a=0 }^{N} U^a B_a^p(x)
\end{equation}
satisfying 
\begin{equation} \label{eq:bho}
\tilde B_h(U, v_h) = 0, \quad \forall \ v_h  \in V_h,
\end{equation}
where
$
\tilde B_h(w, v) = \tilde a_h(w, v) - \lambda \tilde b_h(w, v).
$

\subsection{Quadratic elements} \label{sec:quadraticelem} 
We consider the $C^1$ quadratic B-spline basis function of the isogeometric analysis (for quadratic finite elements, we refer to \cite{ainsworth2004discrete,ainsworth2010optimally}). Applying quadrature rule $G_3$, one obtains the following equation for the value $U^a$ of the approximation at node $x_a = ah$ where $a =3, 4, \cdots, N - 2$ (that is, a node away from boundary),
 \begin{equation} \label{eq:quadraticg3}
 \begin{aligned}
( 20 + \Lambda^2 ) ( U^{a-2} + U^{a+2})+ ( 40 + 26 \Lambda^2 ) ( U^{a-1} + U^{a+1} ) - (120 - 66 \Lambda^2 ) U^a=0,
 \end{aligned}
 \end{equation}
 where $\Lambda = \sqrt{\lambda} h$. 
 Let $i^2 = -1$. This equation admits nontrivial solutions of the form 
 \begin{equation} \label{eq:1dbsol}
U^a = e^{ia \mu^{(2)}_{G_3}h}
\end{equation}
provided that $\mu^{(2)}_{G_3}$ satisfies 
 \begin{equation} \label{eq:quadraticg3dr}
 \begin{aligned}
( 20 + \Lambda^2 ) \cos^2(\mu^{(2)}_{G_3}h) + ( 20 + 13 \Lambda^2 ) \cos(\mu^{(2)}_{G_3}h) + (-40 + 16 \Lambda^2 ) = 0,
 \end{aligned}
 \end{equation}
which is known as the \textit{discrete dispersion relation} (c.f., \cite{ainsworth2004discrete}) for the discrete method with a particular quadrature rule. Solving \eqref{eq:quadraticg3dr} for $\mu^{(2)}_{G_3}h$ and writing the expression as a series in $\Lambda$ (assuming $\Lambda = \sqrt{\lambda} h<1$), we obtain the \textit{discrete dispersion error} 
\begin{equation} \label{eq:quadraticg3dre}
\mu^{(2)}_{G_3}h = \Lambda - \frac{1}{1440} \Lambda^5 - \frac{1}{6720} \Lambda^7 + \mathcal{O}(\Lambda^9).
\end{equation}
 
Now applying $GL_3$, one obtains
 \begin{equation} \label{eq:quadraticgl3}
 \begin{aligned}
( 16 + \Lambda^2 ) ( U^{a-2} + U^{a+2})+ ( 32 + 20 \Lambda^2 ) ( U^{a-1} + U^{a+1} ) - (96 - 54 \Lambda^2 ) U^a=0,
 \end{aligned}
 \end{equation}
 which leads to the discrete dispersion relation
 \begin{equation} \label{eq:quadraticgl3dr}
 \begin{aligned} 
( 16 + \Lambda^2 ) \cos^2(\mu^{(2)}_{GL_3}h) + ( 16 + 10 \Lambda^2 ) \cos(\mu^{(2)}_{GL_3}h) + (-32 + 13 \Lambda^2 ) = 0.
 \end{aligned}
 \end{equation}
Solving \eqref{eq:quadraticgl3dr} for $\mu^{(2)}_{GL_3}h$ and writing the expression as a series in $\Lambda$, we obtain
\begin{equation} \label{eq:quadraticgl3dre}
\mu^{(2)}_{GL_3}h = \Lambda + \frac{1}{2880} \Lambda^5 - \frac{1}{16128} \Lambda^7 + \mathcal{O}(\Lambda^9).
\end{equation}

Lastly, we apply $G_2$ to obtain 
 \begin{equation} \label{eq:quadraticg2}
 \begin{aligned}
( 24 + \Lambda^2 ) ( U^{a-2} + U^{a+2})+ ( 48 + 32 \Lambda^2 ) ( U^{a-1} + U^{a+1} ) - (144 - 78 \Lambda^2 ) U^a=0,
 \end{aligned}
 \end{equation}
 which leads to the discrete dispersion relation
 \begin{equation} \label{eq:quadraticg2dr}
 \begin{aligned} 
( 24 + \Lambda^2 ) \cos^2(\mu^{(2)}_{G_2}h) + ( 24 + 16 \Lambda^2 ) \cos(\mu^{(2)}_{G_2}h) + (-48 + 19 \Lambda^2 ) = 0.
 \end{aligned}
 \end{equation}
Solving \eqref{eq:quadraticg2dr} for $\mu^{(2)}_{G_2}h$ and writing the expression as a series in $\Lambda$, we obtain
\begin{equation} \label{eq:quadraticg2dre}
\mu^{(2)}_{G_2}h = \Lambda - \frac{1}{720} \Lambda^5 - \frac{5}{24192} \Lambda^7 + \mathcal{O}(\Lambda^9).
\end{equation}

\begin{remark} \label{rem:quadratic}
$G_3$ integrates both the stiffness and mass terms exactly while $GL_3$ and $G_2$ integrate the stiffness terms exactly but under-integrate the mass ones. However, all of them yield the optimal order $h^{2p}$ of convergence. Similar calculations also indicate that once the stiffness term is  under-integrated, for example by $GL_2$ or $G_1$, the optimal convergence is lost, which is expected in the view of Strang's second lemma. The impact of under-integration on convergence is discussed in \cite{strang1973analysis} among others.
\end{remark}

The leading coefficient $-\frac{1}{1440}$ in \eqref{eq:quadraticg3dre} and the leading coefficient $\frac{1}{2880}$ in \eqref{eq:quadraticgl3dre} allows us to conjecture that the blending rule $\frac{1}{3} G_3 + \frac{2}{3} GL_3$ eliminates the leading term in the dispersion error since there holds $\frac{1}{3} \cdot (-\frac{1}{1440}) + \frac{2}{3} \cdot \frac{1}{2880} = 0.$   
Hence, the signs and coefficients of \eqref{eq:quadraticg3dre}, \eqref{eq:quadraticgl3dre}, and \eqref{eq:quadraticg2dre} allow us to propose the following blendings
\begin{equation} \label{eq:quadraticg3gl3}
\frac{1}{3} G_3 + \frac{2}{3} GL_3, \qquad 2 G_3 -  G_2, \qquad \text{and} \qquad \frac{4}{5} GL_3 + \frac{1}{5}  G_2.
\end{equation}
All of them lead to the discrete dispersion relation for the optimal scheme
 \begin{equation} \label{eq:quadratico2dr}
 \begin{aligned} 
( 120 + 7\Lambda^2 ) \cos^2(\mu^{(2)}_{O_2}h) + ( 120 + 76 \Lambda^2 ) \cos(\mu^{(2)}_{O_2}h) + (-240 + 97 \Lambda^2 ) = 0,
 \end{aligned}
 \end{equation}
which gives the optimal dispersion error
\begin{equation} \label{eq:quadraticoptdre}
\mu^{(2)}_{O_2}h = \Lambda - \frac{11}{120960} \Lambda^7 - \frac{1}{345600} \Lambda^9 + \mathcal{O}(\Lambda^{11}).
\end{equation}

\begin{remark} \label{rem:quadratic2}
The blending rules allow us to eliminate the fifth order error and give a seventh order error. Moreover, the coefficient of the seventh order error is the sum of the corresponding ones in \eqref{eq:quadraticg3dre}, \eqref{eq:quadraticgl3dre}, and \eqref{eq:quadraticg2dre} with weights specified in \eqref{eq:quadraticg3gl3}. However, calculation shows that this is not the case for the coefficients of the ninth or higher orders. 
\end{remark}

To show that the blending is optimal, applying the general blending rule
\begin{equation}
\mathcal{Q}_\tau = \tau G_3 + (1- \tau) GL_3
\end{equation}
gives the dispersion error
\begin{equation} 
\mu^{(2)}_{\mathcal{Q}_\tau}h = \Lambda - \frac{-1 + 3\tau }{2880} \Lambda^5 - \frac{5+7\tau}{80640} \Lambda^7 + \mathcal{O}(\Lambda^9),
\end{equation}
where $\tau = 1/3$ eliminates the fifth order error term, hence it is the optimal blending parameter. Similarly, one can verify the optimal blending parameters in \eqref{eq:quadraticg3gl3}. 
Alternatively, one can optimize the dispersion error directly from the error expansion in the most general form for the mass
\begin{equation} \label{eq:quadraticgf}
 \begin{aligned}
-\frac{1}{6} ( U^{a-2} + U^{a+2}) - \frac{1}{3} ( U^{a-1} + U^{a+1} ) +  U^a - \Lambda^2 \big( \alpha ( U^{a-2} + U^{a+2}) & \\
+ \beta ( U^{a-1} + U^{a+1} ) + (1 - 2\alpha - 2\beta) U^a  \big) & = 0,
 \end{aligned}
 \end{equation}
 where $\alpha$ and $\beta$ are parameters representing the approximated mass entries. Herein, partition of unity requires the coefficient of $U^a$ in the mass term in \eqref{eq:quadraticgf} to be $(1 - 2\alpha - 2\beta)$.
This leads to a dispersion error expansion 
\begin{equation} \label{eq:quadraticode}
\begin{aligned} 
\mu^{(2)}_{O_2}h  = & \Lambda + \left(-2 \alpha-\frac{\beta}{2}+\frac{1}{8}\right) \Lambda^3 \\
& +\frac{ \left(34560 \alpha^2+17280 \alpha \beta-3360 \alpha+2160 \beta^2-1560\beta + 227\right)}{5760} \Lambda^5  \\
& + \frac{\Lambda^7 }{64512} \Big(-1290240 \alpha^3-967680 \alpha^2 \beta+134400 \alpha^2-241920 \alpha \beta^2+147840 \alpha \beta \\
& -17808 \alpha-20160 \beta^3+28560 \beta^2-10164 \beta+1039\Big) + \mathcal{O}( \Lambda^9).
\end{aligned}
\end{equation}

\noindent
To optimize the dispersion error, we set and solve
\begin{equation}
\begin{aligned}
-2 \alpha-\frac{\beta}{2}+\frac{1}{8} & = 0 \\
34560 \alpha^2+17280 \alpha \beta-3360 \alpha+2160 \beta^2-1560\beta + 227 & = 0
\end{aligned}
\end{equation}
to obtain
\begin{equation}
\alpha = \frac{7}{720}, \qquad \beta = \frac{19}{90},
\end{equation}
which in return simplifies \eqref{eq:quadraticode} to the optimized dispersion error expression
\begin{equation} 
\mu^{(2)}_{O_2}h = \Lambda - \frac{11}{120960} \Lambda^7 + \mathcal{O}(\Lambda^{9}).
\end{equation}

Thus, this alternative way leads to the same optimized dispersion error. In fact, values of $\alpha$ and $\beta$ uniquely determine the coefficients in the blending schemes \eqref{eq:quadraticg3gl3}.

\subsection{Cubic elements} \label{sec:cubicelem} 
Now we consider the cubic B-spline element case with maximum continuity, that is, $C^2$ basis functions. Similarly, applying $G_4, GL_4$, and $G_3$, we obtain the following discrete dispersion relations
\begin{equation} \label{eq:cubicg4dr}
 \begin{aligned}
0 & = (42  + \omega^2h^2  ) \cos^3(\mu^{(3)}_{G_4} h)  + (504 + 60\omega^2h^2 ) \cos^2(\mu^{(3)}_{G_4} h) \\
& \quad + (126 + 297 \omega^2h^2)  \cos(\mu^{(3)}_{G_4} h) + (-672 + 272 \omega^2h^2),\\
0 & = (90  + 2\omega^2h^2 ) \cos^3(\mu^{(3)}_{GL_4} h) + (1080 + 129\omega^2h^2 ) \cos^2(\mu^{(3)}_{GL_4} h) \\ 
& \quad + (270 + 636 \omega^2h^2)  \cos(\mu^{(3)}_{GL_4} h) + (-1440 + 583 \omega^2h^2),\\
0 &= (120 + 3\omega^2h^2 ) \cos^3(\mu^{(3)}_{G_3} h) + (1440 + 171\omega^2h^2 ) \cos^2(\mu^{(3)}_{G_3} h) \\ 
&\quad + (360 + 849 \omega^2h^2)  \cos(\mu^{(3)}_{G_3} h) + (-1920 + 777 \omega^2h^2). \\
 \end{aligned}
 \end{equation}
They lead to the dispersion error expansions given below
\begin{equation} \label{eq:cubicg4dre}
 \begin{aligned}
\mu^{(3)}_{G_4}h & = \Lambda - \frac{1}{60480} \Lambda^7 - \frac{1}{907200} \Lambda^9 + \mathcal{O}(\Lambda^{11}), \\
\mu^{(3)}_{GL_4}h & = \Lambda - \frac{1}{100800} \Lambda^7 - \frac{11}{1814400} \Lambda^9 + \mathcal{O}(\Lambda^{11}), \\
\mu^{(3)}_{G_3}h & = \Lambda - \frac{13}{604800} \Lambda^7 - \frac{37}{7257600} \Lambda^9 + \mathcal{O}(\Lambda^{11}).
 \end{aligned}
\end{equation}

Following the procedure for the quadratic case, we obtain the corresponding optimal blending schemes 
\begin{equation}
-\frac{3}{2} G_4  + \frac{5}{2} GL_4, \qquad \frac{13}{3} G_4 - \frac{10}{3} G_3, \qquad \text{and} \qquad \frac{13}{7} GL_4 - \frac{6}{7} G_3,
\end{equation}
which all result in the following optimized dispersion error expression
\begin{equation} \label{eq:cubicoptdre}
\mu^{(3)}_{O_3}h = \Lambda - \frac{1}{145152} \Lambda^9 + \frac{19}{68428800} \Lambda^{11} + \mathcal{O}(\Lambda^{13}).
\end{equation}

\begin{remark} \label{rem:cubic}
As in the quadratic case, one can assign $\alpha, \beta$, and $\gamma$ to the mass terms to derive the optimized dispersion error expression. There are other blending schemes which give the same optimized dispersion error expression. Here, we list a few
\begin{equation} 
\begin{aligned}
& \frac{4}{35} GL_3 + \frac{36}{35} G_2 - \frac{1}{7} GL_2, && \qquad \frac{10}{49} G_3 + \frac{234}{245} G_2 - \frac{39}{245} GL_2, \\
& \frac{20}{7} G_3 - \frac{52}{35} GL_3 - \frac{13}{35} GL_2, && \qquad \frac{10}{7} GL_4 - \frac{12}{35} GL_3 - \frac{3}{35} GL_2. \\
\end{aligned}
\end{equation}
Herein, blending more quadrature rules does not reduce the dispersion and eigenvalue errors further.
\end{remark}

\begin{table}[ht]
\centering
\begin{tabular}{ c | c  c  c  c }
$p$ & $G_{p+1}$ & $GL_{p+1}$  & $G_p$ & $O_p$ \\[0.1cm]  \hline 
4 & $\dfrac{3}{20 \cdot 9!}$ & $\dfrac{79}{560 \cdot 9!}$ & $\dfrac{11}{70 \cdot 9!}$ & $\dfrac{317}{24\cdot 11!}$  \\[0.3cm]  
5 & $\dfrac{5}{12 \cdot 11!}$ & $\dfrac{29}{70 \cdot 11!}$ & $\dfrac{211}{504 \cdot 11!}$ & $\dfrac{35039}{420\cdot 13!}$ \\[0.3cm] 
6 & $\dfrac{691}{420 \cdot 13!}$ & $\dfrac{91177}{55440 \cdot 13!}$ & $\dfrac{5069}{3080 \cdot 13!} $ &$\dfrac{15479}{24\cdot 15!}$ \\[0.3cm] 
7 & $\dfrac{35}{4 \cdot 15!} $ & $ \dfrac{105103}{12012 \cdot 15!}$ & $ \dfrac{60061}{6864 \cdot 15!}$ & $\dfrac{91067}{15\cdot17!}$ \\[0.3cm]  
\end{tabular}
\caption{Leading order coefficients of the discrete dispersion relations for the different quadratures described and for different polynomial orders $p=4,5,6,7$.}
 \label{tab:4-7}
\end{table}

\subsection{Higher-order elements} \label{sec:hoelem} 
For higher order elements, the calculations and derivations become more complicated. In the following, we list our results up to order seven for isogeometric analysis  with maximum continuity basis functions. As before, for orders $p=4,5,6,7$, we apply $G_{p+1}, GL_{p+1}$, and $G_p$ to compute the integrals. Their corresponding dispersion error expressions are of the form
\begin{equation} \label{eq:no}
\mu^{(p)}_{Q}h = \Lambda - \epsilon^Q_p \Lambda^{2p+1} + \mathcal{O}(\Lambda^{2p+3}),
\end{equation}
for $Q=G_{p+1}, GL_{p+1}, G_p$, while the optimized dispersion expressions are of the form
\begin{equation} \label{eq:o}
\mu^{(p)}_{O_p}h = \Lambda - \epsilon_p^o \Lambda^{2p+3} + \mathcal{O}(\Lambda^{2p+5}).
\end{equation}

In Table \ref{tab:4-7}, we list the coefficients $\epsilon^Q_p$ and $\epsilon_p^o$ for $p=4,5,6,7$. As before, the optimal blending schemes are not unique; we list a few of them below
\begin{equation} 
\begin{aligned}
p & = 4 \qquad  \quad -\frac{79}{5} G_5 + \frac{84}{5} GL_5, && \qquad 22G_5 - 21 G_4, \\
p & = 5  \qquad  \quad  -174G_6 + 175GL_6, && \qquad 211G_6 -210G_5, \\
p & = 6 \qquad  \quad  -\frac{91177}{35} G_7 + \frac{91212}{35} GL_7, && \qquad \frac{30414}{10} G_7 - \frac{30404}{10} G_6, \\
p & = 7 \qquad  \quad  -\frac{105103}{2} G_8 + \frac{105105}{2} GL_8, && \qquad 60061G_8 - 60060G_7.
\end{aligned}
\end{equation}

These blending schemes can be rewritten as 
\begin{equation} \label{eq:blending}
G_{p+1} + C_{1,p} \cdot (GL_{p+1} - G_{p+1}) \qquad \text{and} \qquad G_{p+1} + C_{2,p} \cdot (G_{p+1} - G_{p}),
\end{equation}
where $C_{1,p}$ and $C_{2,p}$ depend on $p$ and satisfy
\begin{equation} \label{eq:2br}
pC_{2,p} - (p+1) C_{1,p} = 0.
\end{equation}
Thus, once the general form of one of them is obtained, the other one can be derived from \eqref{eq:2br}.

\begin{remark} \label{rem:high}
The analysis of the possible generalization of the optimal blending schemes to arbitrary order $p$ with variable continuities $C^k, k=1,\cdots,p-1,$ is an open question and is still under investigation. 
\end{remark}

\subsection{Extension to multiple dimensions} \label{sec:md} 
The extension is similar to the extension done for finite elements in \cite{ainsworth2010optimally}. The conclusion is that the optimal blending for arbitrary dimension coincides with the one we derive above for the one-dimensional case and is independent of the number of spatial dimensions.

In multiple dimensions, a $d$-dimensional basis function is a tensor product of one-dimensional basis functions, that is,
$
\prod_{k=1}^d B_{n_k}
$ 
(we drop $p$ here as the following derivation is independent of the polynomial order $p$).
We seek a solution of the form 
\begin{equation} \label{eq:mdsol}
U(\bfs{x}) = \prod_{k=1}^d U_k(x_k),
\end{equation}
where each $U_k(x_k)$ has a one-dimensional representation \eqref{eq:1dsol}, which is further written by using \eqref{eq:1dbsol} as
\begin{equation}
U_k(x_k) = \sum_{n_k =0}^{N_k} e^{in_k \sqrt{\lambda_k} h_k} B_{n_k}
\end{equation}
with $\lambda_k$ (only in this subsection) being the squared wave number in the dimension $k=1, \cdots, d$. 
This allows us to also write 
\begin{equation} \label{eq:mdbsol}
\begin{aligned}
U(\bfs{x}) & = \sum_{n_1, \cdots, n_d} \Big( e^{in_1 \sqrt{\lambda_1} h_1 + \cdots + in_d \sqrt{\lambda_d} h_d} \prod_{k=1}^d B_{n_k} \Big) \\
& =  \sum_{n_1, \cdots, n_d} \Big( U_{n_1, \cdots, n_d} \prod_{k=1}^d B_{n_k} \Big),
\end{aligned}
\end{equation}
where $U_{n_1, \cdots, n_d}$ with $n_k = 0,1, \cdots, N_k, k=1,\cdots, d$ are the coefficients of the linear combination of the  multi-dimensional basis functions.

From \eqref{eq:mdsol}, we have
\begin{equation} \label{eq:mdb}
\begin{aligned}
\tilde b_h\big( U(\bfs{x}), \prod_{k=1}^d B_{n_k} \big) = \prod_{k=1}^d \tilde b_h\big(U_k(x_k), B_{n_k} (x_k) \big).
\end{aligned}
\end{equation}

\noindent
Suppose that in each dimension $k$, $U_k(x_k)$ satisfies \eqref{eq:bho}, then
\begin{equation} \label{eq:1dbho}
\tilde a_h\big(U_k(x_k), B_{n_k} (x_k) \big) - \lambda_k \tilde b_h\big(U_k(x_k), B_{n_k} (x_k) \big) = 0.
\end{equation}
By applying the one-dimensional equations \eqref{eq:1dbho} with $k=1,\cdots,d$, we obtain
\begin{equation} \label{eq:mda}
\begin{aligned}
\tilde a_h\big( U(\bfs{x}), \prod_{k=1}^d B_{n_k} \big) & = \sum_{l=1}^d \Big( \tilde a_h\big( U_l (x_l), B_{n_l} (x_l) \big)  \prod_{k \ne l, k=1}^d \tilde b_h\big(U_k(x_k), B_{n_k} (x_k) \big) \Big) \\
& = \sum_{l=1}^d \Big( \lambda_l  \prod_{k=1}^d \tilde b_h\big(U_k(x_k), B_{n_k} (x_k) \big) \Big) \\
& = \Big( \sum_{l=1}^d  \lambda_l \Big) \Big(\prod_{k=1}^d \tilde b_h\big(U_k(x_k), B_{n_k} (x_k) \big) \Big).
\end{aligned}
\end{equation}

The multi-dimensional problem is to find $U(\bfs{x})$ satisfying for all the basis functions
\begin{equation} \label{eq:md}
\tilde a_h\big( U(\bfs{x}), \prod_{k=1}^d B_{n_k} \big) - \lambda \tilde b_h\big( U(\bfs{x}), \prod_{k=1}^d B_{n_k} \big) = 0.
\end{equation}

Plugging \eqref{eq:mdb} and \eqref{eq:mda} into \eqref{eq:md}, the multidimensional problem admits a nontrivial solution of the form $\eqref{eq:mdbsol}$ provided that 
\begin{equation} \label{eq:1dmd}
\lambda = \sum_{k=1}^d  \lambda_k.
\end{equation}

Replacing the $\lambda$ in the expression \eqref{eq:bho} for one dimension and \eqref{eq:md} for multiple dimensions with the numerical approximated $\tilde \lambda^h_k$ and $\lambda^h$, respectively, the same derivations lead to
 \begin{equation} \label{eq:1dmdh}
\tilde \lambda^h = \sum_{k=1}^d  \lambda_k^h.
\end{equation}

Thus, subtracting \eqref{eq:1dmd} from \eqref{eq:1dmdh}, the dispersion error for multidimensional problems consists of the dispersion errors for each dimension, that is,
 \begin{equation} \label{eq:1dmderr}
\tilde \lambda^h - \lambda = \sum_{k=1}^d  ( \tilde \lambda_k^h - \lambda_k ),
\end{equation}
which means that the numerical schemes for the multi-dimensional problem \eqref{eq:pdee} preserve the error estimations of the one-dimensional problem when using the tensor product structure for the mesh discretization.

\subsection{Duality with spectrum analysis} \label{sec:sa} 
There is a symmetry between \eqref{eq:pde} and \eqref{eq:pdee}.  The duality between spectrum analysis and dispersion analysis has been established in \cite{hughes2008duality}. We denote the approximate squared frequencies by $\tilde \lambda^{p,h}_Q$, where $Q$ specifies the quadrature rule. For spectrum analysis in the quadratic case, we obtain
\begin{equation} \label{eq:dual2g3}
\begin{aligned}
\sqrt{\tilde \lambda^{2,h}_{G_3}} h & = \sqrt{ \frac{40 - 20 \cos(\Lambda) - 20 \cos^2(\Lambda) }{16 + 13 \cos(\Lambda) + \cos^2(\Lambda) } }, 
\sqrt{\tilde \lambda^{2,h}_{GL_3}} h  = \sqrt{ \frac{32 - 16 \cos(\Lambda) - 16 \cos^2(\Lambda) }{13 + 10 \cos(\Lambda) + \cos^2(\Lambda) } }, \\
\sqrt{\tilde \lambda^{2,h}_{G_2}} h & = \sqrt{ \frac{48 - 24 \cos(\Lambda) - 24 \cos^2(\Lambda) }{19 + 16 \cos(\Lambda) + \cos^2(\Lambda) } }, 
\sqrt{\tilde \lambda^{2,h}_{O_2}} h = \sqrt{ \frac{120 \big(2 -  \cos(\Lambda) - \cos^2(\Lambda) \big) }{97 + 76 \cos(\Lambda) + 7\cos^2(\Lambda) } }, \\
\end{aligned}
\end{equation}
where the first equation is also given in \cite{hughes2008duality}.

Applying the Taylor expansion on these expressions gives
\begin{equation}
\begin{aligned}
\sqrt{\tilde \lambda^{2,h}_{G_3}} h & = \Lambda + \frac{1}{1440} \Lambda^5 + \mathcal{O}(\Lambda^7), 
\sqrt{\tilde \lambda^{2,h}_{GL_3}}h  = \Lambda - \frac{1}{2880} \Lambda^5 + \mathcal{O}(\Lambda^7), \\
\sqrt{\tilde \lambda^{2,h}_{G_2}}h & = \Lambda + \frac{1}{720} \Lambda^5 + \mathcal{O}(\Lambda^7), 
\sqrt{\tilde \lambda^{2,h}_{O_2}}h  = \Lambda + \frac{11}{120960} \Lambda^7 + \mathcal{O}(\Lambda^9). \\
\end{aligned}
\end{equation}

Similarly, this spectrum analysis can be done for cubic and higher order methods. In general, for multiple dimensions, in the view of \eqref{eq:1dmderr}, we obtain error expressions for $p=2,\cdots, 7$ in the forms
\begin{equation} \label{eq:sno}
\sqrt{\tilde \lambda^{p,h}_{Q}} h = \Lambda + \epsilon^Q_p \Lambda^{2p+1} + \mathcal{O}(\Lambda^{2p+3})
\end{equation}
for $Q=G_{p+1}, GL_{p+1}, G_p$ and
\begin{equation} \label{eq:so}
\sqrt{\tilde \lambda^{p,h}_{O_p}} h = \Lambda + \epsilon_p^o \Lambda^{2p+3} + \mathcal{O}(\Lambda^{2p+5})
\end{equation}
for the optimal schemes. For the cubic case, the coefficients $\epsilon_3^o = 1/145152$ and $\epsilon^Q_3 = 1/60480$, $1/100800, 13/604800$ for $Q = G_3, GL_3$, and $G_2$, respectively, while for $p=4$, $5,6,7$, coefficients $\epsilon^Q_p$ and $\epsilon_p^o$ are given in Table \ref{tab:4-7}.

\begin{remark} \label{rem:sa}
Equations \eqref{eq:sno} with \eqref{eq:no} and \eqref{eq:so} with \eqref{eq:o} reveal the duality principle of dispersion and spectrum analysis in the error expansion form.  The duality principle remains valid for the optimal blending schemes. The different signs of the coefficients $\epsilon^Q_p$ and $\epsilon_p^o$ in the error expressions are consequences of duality. 
\end{remark}

\section{Error analysis} \label{sec:ea}
In the framework of finite element analysis, the eigenvalue and eigenfunction errors for both $(\lambda_j^h, u_j^h)$ and $(\tilde \lambda_j^h, \tilde u_j^h)$ are of the forms \eqref{eq:femeverr}, c.f., \cite{strang1973analysis, banerjee1989estimation}. 
We now estimate the errors of the approximated eigenpair $(\tilde \lambda_j^h, \tilde u_j^h)$ under the framework of isogeometric analysis.

\subsection{Eigenvalue estimates} \label{sec:evana}
The eigenvalue errors when using the standard quadrature rules, such as Gauss and Lobatto rules, converge at rate $h^{2p}$, c.f., \cite{banerjee1992note,cottrell2006isogeometric,hughes2014finite}. This is also confirmed by our theoretical finding \eqref{eq:sno} for orders $p=2, \cdots, 7$ and we conjecture that this is true for arbitrary order. To see this, squaring both sides of \eqref{eq:sno} and using $\Lambda = \sqrt{\lambda} h$ gives the following estimate
\begin{equation*}
\tilde \lambda^{p,h}_{Q} h^2 = \lambda h^2 + 2 \epsilon^Q_p \lambda^{p+1} h^{2p+2} + \mathcal{O}(\lambda^{p+2} h^{2p+4}),
\end{equation*}
which reduces to 
\begin{equation*}
|\tilde  \lambda^{p,h}_{Q} - \lambda | \le C h^{2p} \lambda^{p+1},
\end{equation*}
where $C$ is a constant independent of $h$. Clearly, this yields the $2p$ order convergence for the eigenvalues. Now, for isogeometric analysis with order up to $p=7$, we have the following eigenvalue estimate

\begin{theorem} \label{thm:eve}
For a fixed $j \geq 1$ and $p=1, \cdots, 7$, assume $\tilde \lambda_j^h = \tilde \lambda^{p,h}_{O_p}$.  For $h$ such that $\max{\{\sqrt{\lambda_j} h, \sqrt{\tilde \lambda_j^h} h \}  } < 1$, we have 
\begin{equation}
| \tilde \lambda_j^h - \lambda_j | \le C h^{2p+2} \lambda_j^{p+2},
\end{equation}
where $C$ is a constant independent of $h$.
\end{theorem}

\begin{proof}
For order $p=1$, this is shown in \cite{ainsworth2010optimally}. For $p=2, \cdots, 7,$ we have the error representation \eqref{eq:so}. For a fixed $j$, squaring both sides of \eqref{eq:so} and using $\Lambda = \sqrt{\lambda_j} h$ gives the following estimate
\begin{equation*}
\tilde \lambda^{p,h}_{O_p} h^2 = \lambda_j h^2 + 2 \epsilon^{o}_p \lambda_j^{p+2} h^{2p+4} + \mathcal{O}(\lambda_j^{p+3} h^{2p+6}),
\end{equation*}
which reduces to the desired result.
\end{proof}

This theorem shows that the optimally-blended schemes produce two extra orders of convergence for the eigenvalue errors.

\subsection{Eigenfunction estimates} \label{sec:efana}
In this section, we establish the optimal convergence rates of the eigenfunction errors. First, we establish the coercivity of the bilinear forms \eqref{eq:ba} and \eqref{eq:bb}.

\begin{lemma}\label{lem:coe}
Given the blending scheme of the form \eqref{eq:blending}, the bilinear forms \eqref{eq:ba} and \eqref{eq:bb} are coercive, that is,  there holds
\begin{equation}
\tilde a_h(v_h, v_h) \ge \tilde \alpha |v_h|^2_{1, \Omega}  \quad \text{and} \quad \tilde b_h(v_h, v_h) \ge \tilde \beta \| v_h \|^2_{0, \Omega}, \quad \forall \ v_h \ \in V_h,
\end{equation}
where $\tilde \alpha, \tilde \beta >0$ are constants independent of $h$. 
\end{lemma}
\begin{proof}

Without loss of generality, we consider the blending scheme $G_{p+1} + C_{2,p} \cdot (G_{p+1} - G_{p})$ for $p$-th order isogeometric analysis. Since $G_{p+1}$ and $G_{p}$ are exact for polynomial spaces of order $2p+1$ and $2p-1$, respectively, the blending scheme is exact for polynomial space of order $2p-1$. Hence, by Theorem 4.1.2 in \cite{ciarlet1978finite}, the bilinear form $\tilde a_h(\cdot, \cdot)$ is coercive.

Now, for $\tilde b_h(v_h, v_h)$, if the quadrature rule applied to the integral is $G_{p+1}$, then $\tilde b_h(v_h, v_h) = \| v_h \|^2_{0, \Omega}.$ In the following, we treat the quadrature rules as operators on integrals and let $\mathcal{I}$ be the identity operator. We calculate
\begin{equation*}
\begin{aligned}
\tilde b_h(v_h, v_h) & = \big( G_{p+1} + C_{2,p} \cdot (G_{p+1} - G_{p}) \big) \circ \int_\Omega v_h^2 \ \text{d} \boldsymbol{x}  \\
& = \big( G_{p+1} \big) \circ \int_\Omega v_h^2 \ \text{d} \boldsymbol{x}  + C_{2,p} \cdot \big( G_{p+1} - G_{p} \big) \circ \int_\Omega v_h^2 \ \text{d} \boldsymbol{x}  \\
& = \| v_h \|^2_{0, \Omega} + C_{2,p} \cdot \big( G_{p+1} - \mathcal{I} \big) \circ \int_\Omega v_h^2 \ \text{d} \boldsymbol{x} + C_{2,p} \cdot \big( \mathcal{I} - G_{p} \big) \circ \int_\Omega v_h^2 \ \text{d} \boldsymbol{x} \\
\end{aligned}
\end{equation*}

By Lemma 3.1 in \cite{banerjee1989estimation}, we have
\begin{equation*}
\begin{aligned}
 \Big| \big( G_{p+1} - \mathcal{I} \big) \circ \int_\Omega v_h^2 \ \text{d} \boldsymbol{x} \Big| & \le C  h^{2p+2} \| v_h \|^2_{p+1, \Omega}, \\
\Big| \big( \mathcal{I} - G_{p} \big) \circ \int_\Omega v_h^2 \ \text{d} \boldsymbol{x} \Big| & \le C  h^{2p} \| v_h \|^2_{p, \Omega},
\end{aligned}
\end{equation*}
where $C$ is a constant independent of $h$.

Thus, for sufficiently small $h$, we have a constant $\tilde \beta > 0$ such that 
\begin{equation*}
\tilde b_h(v_h, v_h) \ge \tilde \beta \| v_h \|^2_{0, \Omega},
\end{equation*}
which completes the proof.
\end{proof}

The boundedness of $\tilde b_h(\cdot, \cdot)$ is  a consequence of the proof. With coercivity, we immediately conclude that the eigenvalues of \eqref{eq:vfho} are positive. 
Before we establish the eigenfunction error estimate, we present the following inequality, which can be obtained by applying the Aubin-Nitsche Lemma (duality argument) on the discrete solution operator $T_h: L^2(\Omega) \to V_h $ defined as $\tilde a_h(T_h (\tilde u^h), v_h) =  \tilde\lambda^h \tilde b_h(\tilde u^h, v_h), \forall \ v_h \in V_h$ as in \eqref{eq:vfho}, c.f., \cite[Theorem 3.2.4]{ciarlet1978finite} or \cite[Section 2.3.4]{ern2013theory}.

\begin{lemma}\label{lem:inv}
Suppose $\Omega = [0, 1]^d \subset \mathbb{R}^d.$ For $u_j, \tilde u_j^h \in H^1_0(\Omega)$, there holds
\begin{equation} \label{eq:s}
\| u_j - \tilde u_j^h \|_{0,\Omega} \le C h | u_j - \tilde u_j^h |_{1,\Omega},
\end{equation}
where $C$ is a constant independent of $h$.
\end{lemma}

\begin{theorem} \label{thm:efe}
For a fixed $j \geq 1$, assume that $u_j$ and $\tilde u_j^h$ are normalized, that is, $b(u_j, u_j)=1$ and $\tilde b_h(\tilde u_j^h, \tilde u_j^h) = 1$, and the signs of eigenfunctions of $u_j$ and $\tilde u_j^h$ are chosen such that $b(u_j, \tilde u_j^h) > 0$. Then for sufficiently small $h$, we have the estimate
\begin{equation}
| u_j - \tilde u_j^h |_{1,\Omega} \le C h^p,
\end{equation}
where $C$ is a constant independent of $h$.
\end{theorem}

\begin{proof}
With the normalization $b(u_j, u_j)=1$ and $\tilde b_h(\tilde u_j^h, \tilde u_j^h) = 1$, by \eqref{eq:vf} and \eqref{eq:vfho} we have
\begin{equation}
a(u_j, u_j) = \lambda_j b(u_j, u_j) = \lambda_j, \qquad \tilde a_h(\tilde u_j^h, \tilde u_j^h) = \tilde \lambda_j^h \tilde b_h(\tilde u_j^h, \tilde u_j^h) = \tilde \lambda_j^h.
\end{equation}

Analogously to the derivations described in  \cite{banerjee1989estimation,solov2013approximation, solov2015finite}, we  estimate the modified bilinear forms, that is, for $v, w \in V_h$ there holds
\begin{equation} \label{eq:al}
\begin{aligned}
| a(v, w) - \tilde a_h(v, w) | & \le C h^{2p} \| v \|_{p,\Omega} \| w \|_{p,\Omega}, \\
| b(v, w) - \tilde b_h(v, w) | & \le C h^{2p} \| v \|_{p,\Omega} \| w \|_{p,\Omega}. \\
\end{aligned}
\end{equation}
Thus, for a fixed $j$, by definition and using \eqref{eq:vf} and \eqref{eq:vfho}, we obtain 
\begin{equation}
\begin{aligned}
| u_j - \tilde u_j^h |_{1,\Omega}^2 & = a(u_j - \tilde u_j^h, u_j - \tilde u_j^h) \\
& = a(u_j, u_j) - 2 a(u_j, \tilde u_j^h) + a(\tilde u_j^h, \tilde u_j^h) \\
& = \lambda_j - 2 \lambda_j b(u_j, \tilde u_j^h) + \tilde \lambda_j^h  + \big( a(\tilde u_j^h, \tilde u_j^h) - \tilde a_h(\tilde u_j^h, \tilde u_j^h) \big) \\
& = \lambda_j \big( 2 - 2 b(u_j, \tilde u_j^h) \big) + \tilde \lambda_j^h - \lambda_j 
+ \big( a(\tilde u_j^h, \tilde u_j^h) - \tilde a_h(\tilde u_j^h, \tilde u_j^h) \big) \\
& = \lambda_j \| u_j - \tilde u_j^h \|_{0,\Omega}^2 + \tilde \lambda_j^h - \lambda_j \\
& \quad + \big( a(\tilde u_j^h, \tilde u_j^h) - \tilde a_h(\tilde u_j^h, \tilde u_j^h) \big)
 + \lambda_j \big( \tilde b_h(\tilde u_j^h, \tilde u_j^h) - b(\tilde u_j^h, \tilde u_j^h) \big).
\end{aligned}
\end{equation}
Rearranging terms and taking absolute value  yields
\begin{equation} \label{eq:4.8}
\begin{aligned}
\Big| | u_j - \tilde u_j^h  |_{1,\Omega}^2 -  \lambda_j \| u_j - \tilde u_j^h \|_{0,\Omega}^2 \Big|  & = \Big| \tilde \lambda_j^h - \lambda_j  +
 \big( a(\tilde u_j^h, \tilde u_j^h) - \tilde a_h(\tilde u_j^h, \tilde u_j^h) \big) \\
& \quad + \lambda_j \big( \tilde b_h(\tilde u_j^h, \tilde u_j^h) - b(\tilde u_j^h, \tilde u_j^h) \big) \Big| \\
& \le | \tilde \lambda_j^h - \lambda_j | + 
 | a(\tilde u_j^h, \tilde u_j^h) - \tilde a_h(\tilde u_j^h, \tilde u_j^h)  |  \\
& \quad + \lambda_j |  \tilde b_h(\tilde u_j^h, \tilde u_j^h) - b(\tilde u_j^h, \tilde u_j^h)  |. \\
\end{aligned} 
\end{equation}
By applying Lemma \ref{lem:inv}, Theorem \ref{thm:eve}, \eqref{eq:al}, and \eqref{eq:4.8}, for sufficiently small $h$, we have a constant $\tilde C > 0$ such that

\begin{equation}
\begin{aligned}
\tilde C | u_j - \tilde u_j^h |_{1,\Omega}^2 & \le (1 - \lambda_j C^2h^2) | u_j - \tilde u_j^h |_{1,\Omega}^2 \\
& \le | u_j - \tilde u_j^h |_{1,\Omega}^2 -  \lambda_j \| u_j - \tilde u_j^h \|_{0,\Omega}^2  \\
& \le \Big| | u_j - \tilde u_j^h |_{1,\Omega}^2 -  \lambda_j \| u_j - \tilde u_j^h \|_{0,\Omega}^2 \Big| \\ 
&\le | \tilde \lambda_j^h - \lambda_j |  + C h^{2p} \\
& \le C h^{2p+2} + C h^{2p} \le C h^{2p}.
\end{aligned}
\end{equation}

Taking the square root on both sides completes the proof.
\end{proof}

\section{Numerical experiments} \label{sec:num}
In this section, we present various numerical examples to illustrate the performance of the optimally-blended rules.
We focus on isogeometric analysis and study the convergence of the eigenvalue (EV) and eigenfunction (EF) errors in one and two dimensions. Three or higher dimensions are simple extensions as discussed in Section \ref{sec:md}. For comparisons with finite element approximations, we refer to \cite{cottrell2006isogeometric,hughes2014finite,puzyrev2017dispersion} where the authors showed that the isogeometric analysis outperforms finite element approximations. 

\begin{figure}[ht]
\centering
\includegraphics[width=12cm]{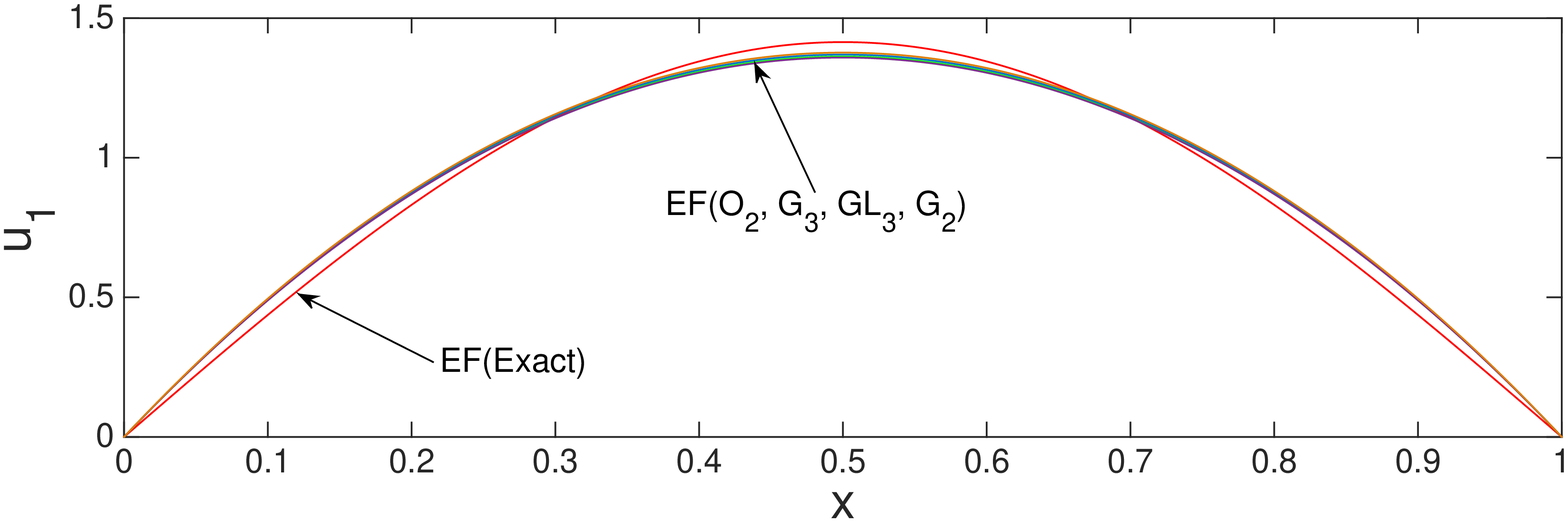} 
\includegraphics[width=12cm]{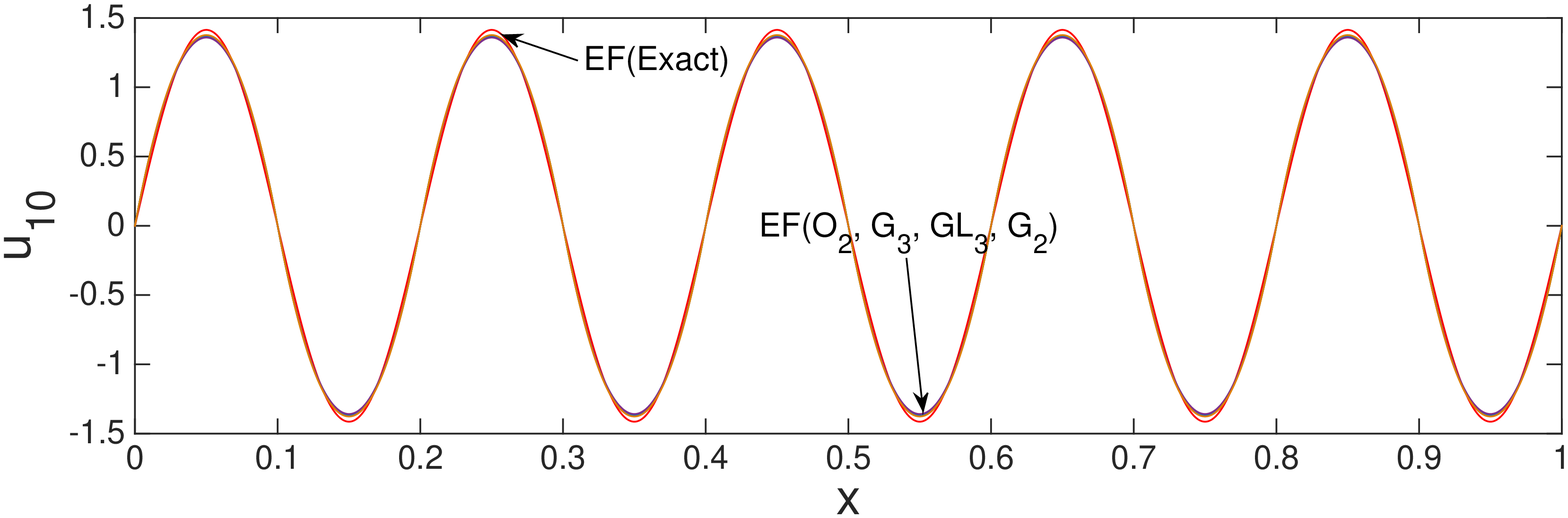}
\caption{Eigenfunctions $u_1$ and $u_{10}$ when using $C^1$ quadratic isogeometric analysis with quadrature rules $G_3, GL_3, G_2$, and $O_2$.}
\label{fig:ef1d2iga}
\includegraphics[width=12cm]{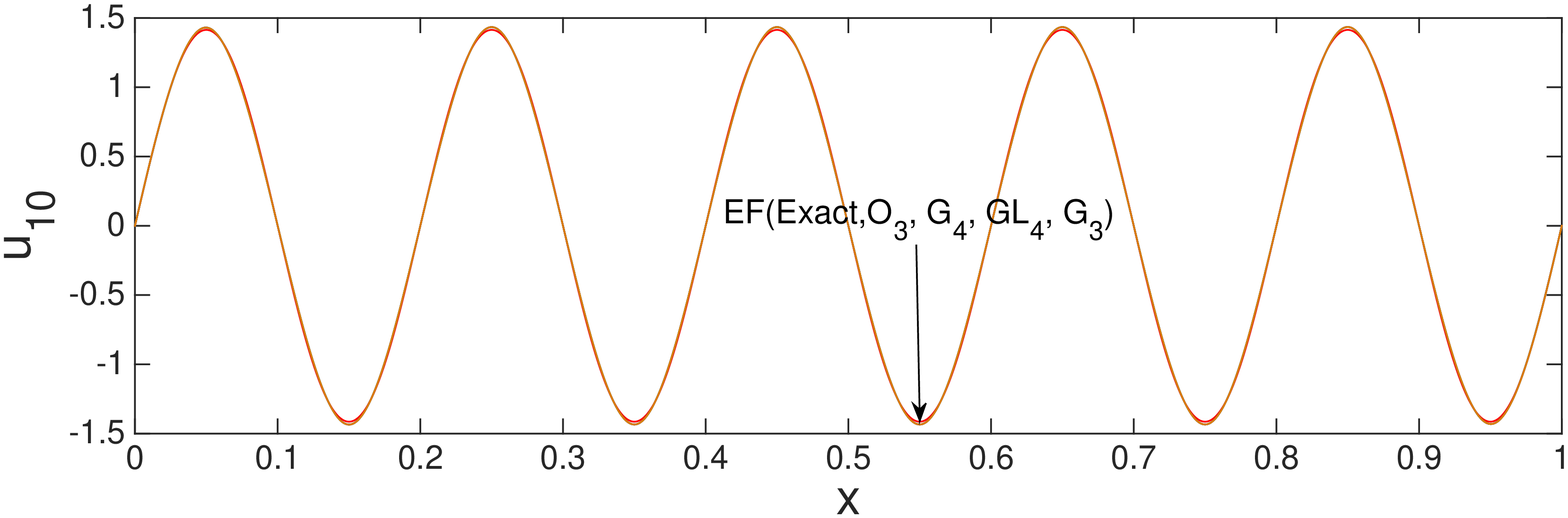}
\caption{Eigenfunction $u_{10}$ when using $C^2$ cubic isogeometric analysis with quadrature rules $G_4, GL_4, G_3$, and $O_3$.}
\label{fig:ef1d3iga}
\end{figure}

\subsection{Convergence study in one dimension}
We consider the classic one-dimensional second-order elliptic eigenvalue problem \eqref{eq:pdee} with eigenvalues and eigenfunctions 
$
\lambda_j = \pi^2 j^2, u_j = \sqrt{2} \sin(j \pi x), j = 1, 2, \cdots.
$
This is the same example studied in \cite{hughes2014finite,puzyrev2017dispersion} among many others. We assume that once the eigenvalue problem is solved, the numerical eigenvalues $\lambda_j^h$ and $\tilde \lambda_j^h$ are sorted in ascending order and paired with the true eigenvalues $\lambda_j$.

First of all, for $p$-th order isogeometric analysis with maximum continuity, that is $C^{p-1}$, the quadrature rule $G_{p+1}$ approximates both stiffness and mass matrices exactly while $GL_{p+1}$, $G_p$, and $O_p$ integrate the stiffness matrices exactly but under-integrate the mass matrices. Despite these differences, all of them lead to accurate approximations to the eigenfunctions. Figure \ref{fig:ef1d2iga} shows the plots of the numerical approximations of eigenfunctions $u_1$ and $u_{10}$ using $C^1$ quadratic isogeometric analysis with two elements for $u_1$ and twenty elements for $u_{10}$. As a comparison, Figure \ref{fig:ef1d3iga} shows the numerical approximation of $u_{10}$ using $C^2$ cubic isogeometric analysis with twenty elements.

\subsubsection{Eigenvalue errors} 
The analysis of eigenvalue errors done in Section \ref{sec:evana} is verified numerically in this subsection. 
The optimally-blended quadrature rules proposed for isogeometric analysis of eigenvalue problem \eqref{eq:pdee} in Section \ref{sec:ddr} yield two additional orders of eigenvalue error convergence.

\begin{figure}[ht]
\centering
\includegraphics[height=5.0cm]{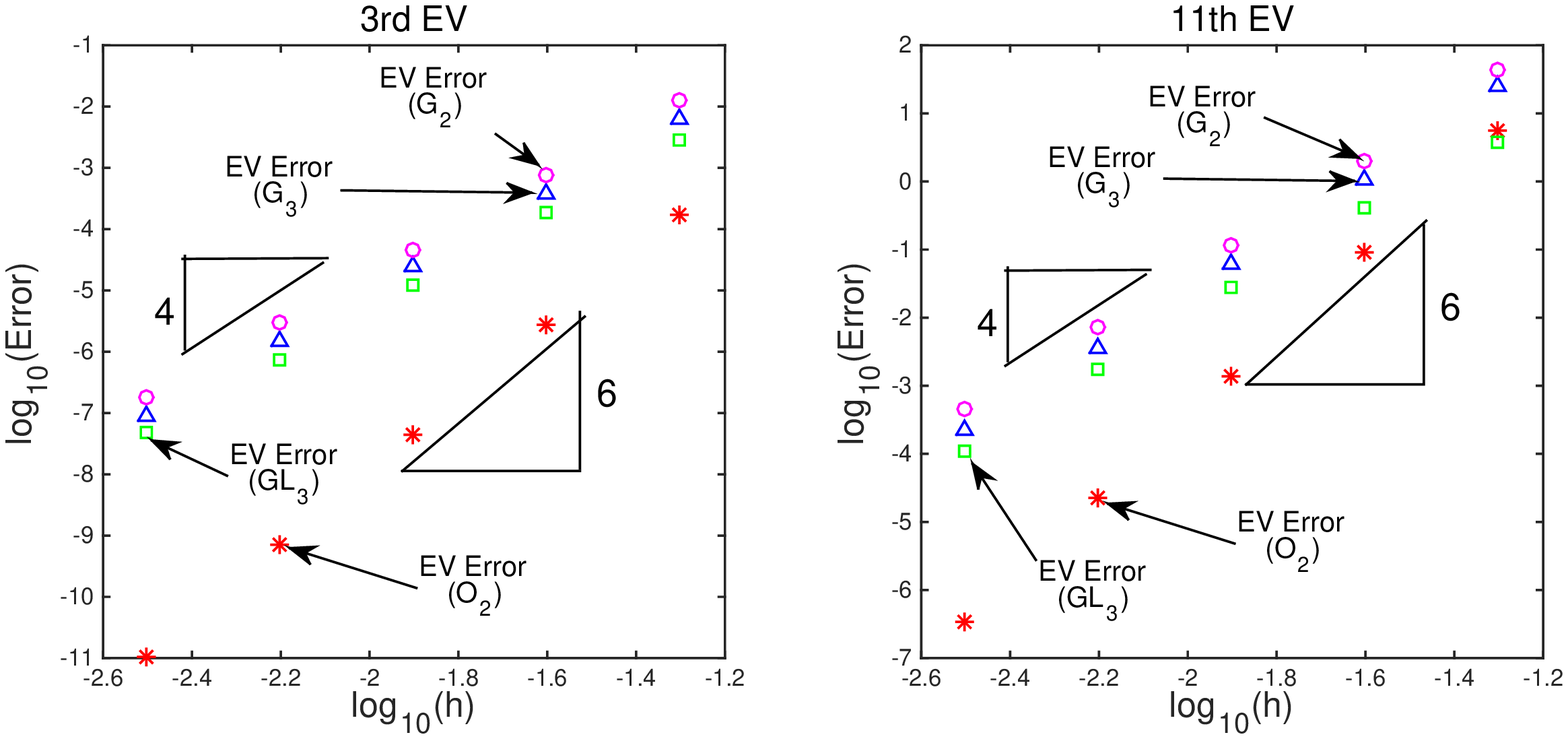} 
\caption{Eigenvalue errors for $\lambda_3$ and $\lambda_{11}$ when using $C^1$ quadratic isogeometric analysis with quadrature rules $G_3, GL_3, G_2$, and $O_2$.}
\label{fig:everr1d2o2iga}
\includegraphics[height=5.0cm]{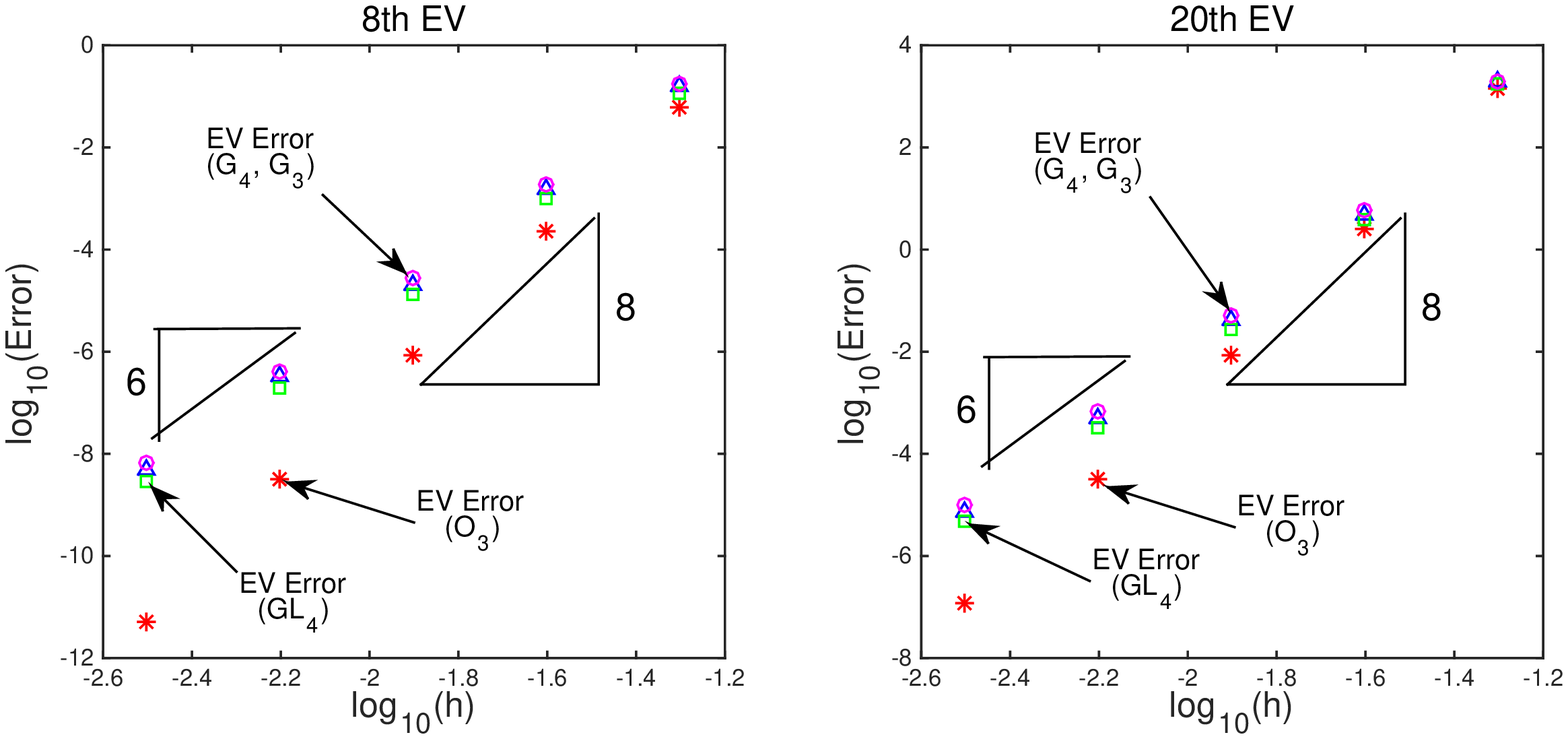} 
\caption{Eigenvalue errors for  $\lambda_8$ and $\lambda_{20}$ when using $C^2$ cubic isogeometric analysis with quadrature rules $G_4, GL_4, G_3$, and $O_3$.}
\label{fig:everr1d3o3iga}
\end{figure}

Figure \ref{fig:everr1d2o2iga} shows the eigenvalue errors when using $C^1$ quadratic isogeometric elements. The domain $\overline\Omega=[0, 1]$ is discretized uniformly with 20, 40, 80, 160, and 320 elements.
The figure shows two extra orders of convergence in the error for the eigenvalues $\lambda_3$ and $\lambda_{11}$. With both stiffness and mass matrices integrated exactly by $G_3$, the eigenvalue errors converge at the rate of $h^4$, while a slight modification of the quadrature rule, which is easily realized by optimally-blended quadrature rule $O_2$, leads to a convergence rate of $h^6$.

\begin{figure}[ht]
\centering
\includegraphics[height=5.0cm]{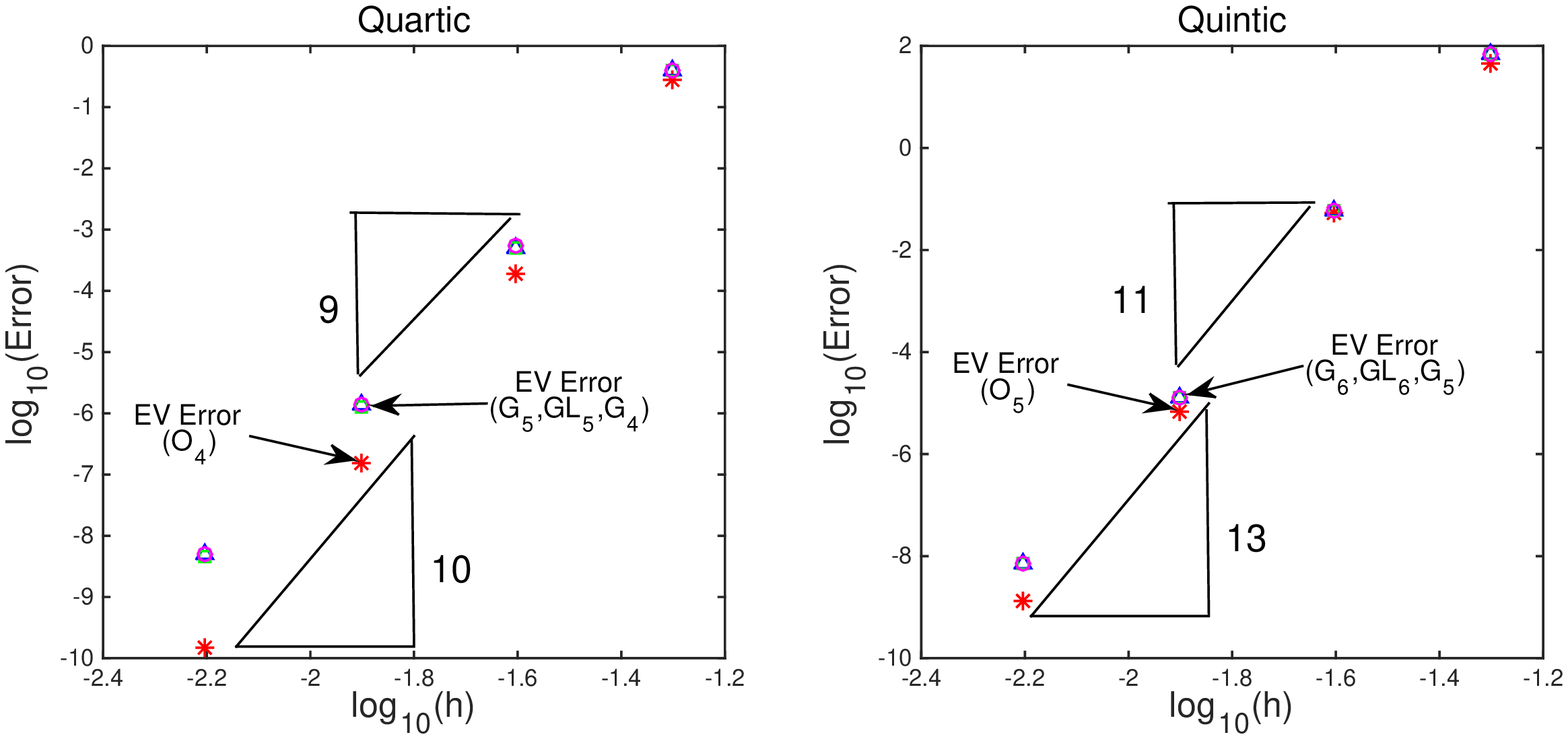} 
\caption{Eigenvalue errors for  $\lambda_{11}$ when using quadrature rules $G_5, GL_5, G_4$, and $O_4$ for $C^3$ quartic isogeometric analysis (left) and $G_6, GL_6, G_5$, and $O_5$ for $C^4$ quintic isogeometric analysis (right).}
\label{fig:everr1d45o45iga}
\end{figure}

\begin{figure}[ht]
\centering
\includegraphics[height=5.0cm]{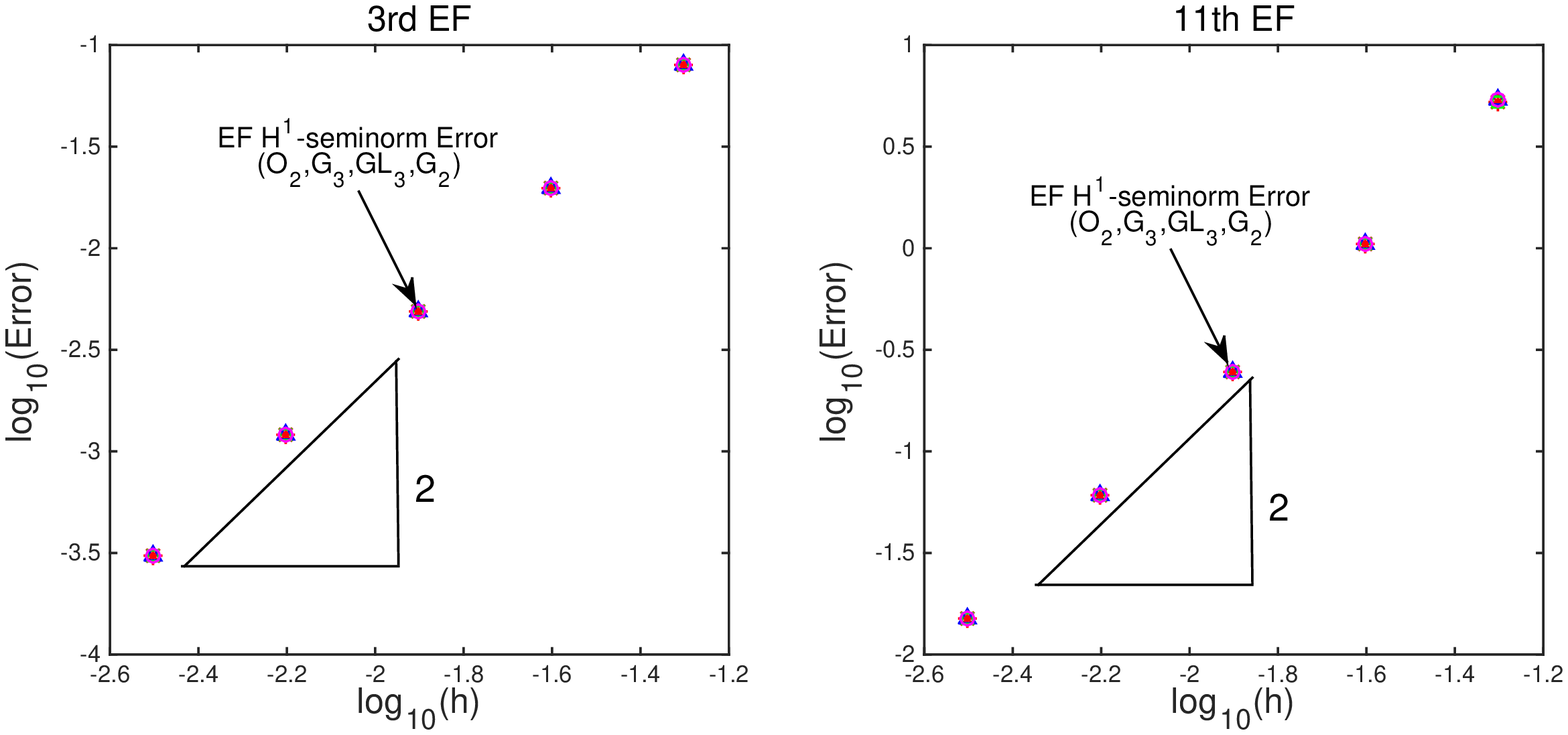} 
\caption{Eigenfunction $H^1$-seminorm errors for $u_3$ and $u_{11}$ when using $C^1$ quadratic isogeometric analysis with quadrature rules $G_3, GL_3, G_2$, and $O_2$.}
\label{fig:eferr1d2igah1}
\includegraphics[height=5.0cm]{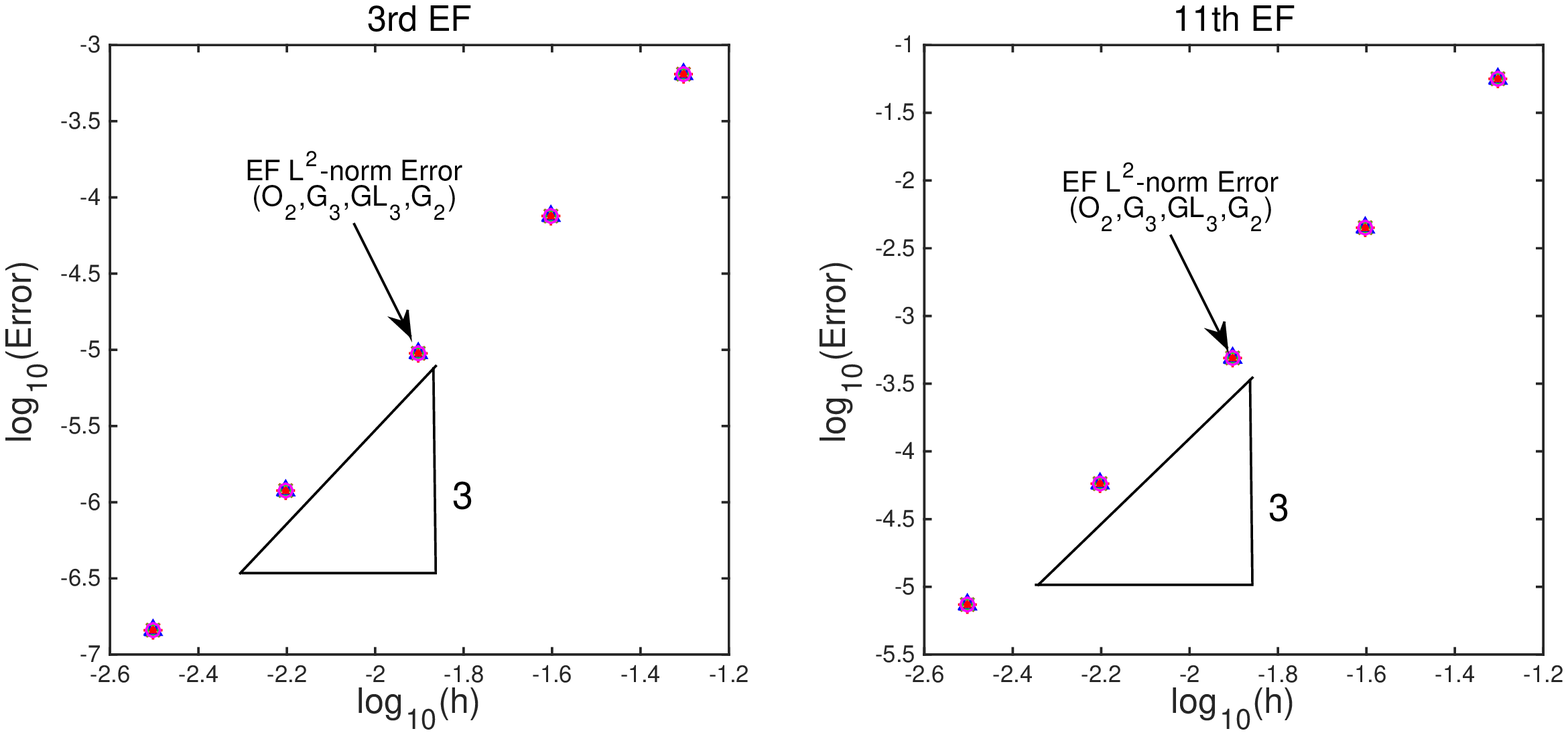} 
\caption{Eigenfunction $L^2$-norm errors for $u_3$ and $u_{11}$ when using $C^1$ quadratic isogeometric analysis with quadrature rules $G_3, GL_3, G_2$, and $O_2$.}
\label{fig:eferr1d2igal2}
\end{figure}

Further inspection of Figure \ref{fig:everr1d2o2iga} reveals that the eigenvalue errors are generally smaller in the case of $GL_3$ than those of $G_3$ and smaller for $G_3$ than for $G_2$. 
This observation confirms the theoretical analysis we present in Section 3 . The leading order terms in the Taylor expansions of the dispersion error \eqref{eq:quadraticgl3dre} for $GL_3$, \eqref{eq:quadraticg3dre} for $G_3$, and \eqref{eq:quadraticg2dre} for $G_2$ have absolute values of 1/2880, 1/1440, and 1/720, respectively. 
Figure \ref{fig:everr1d3o3iga} shows the results when using $C^2$ cubic isogeometric analysis while Figure \ref{fig:everr1d45o45iga} shows these when using $C^3$ quartic and  $C^4$ quintic isogeometric analysis. A slight extra-superconvergence is observed in Figure \ref{fig:everr1d45o45iga} for quartic and quintic isogeometric elements in coarse meshes. These numerical results confirm the eigenvalue error analysis we discuss in Section \ref{sec:evana}.

\subsubsection{Eigenfunction errors} 
In this subsection, we verify numerically the analysis of Section \ref{sec:efana}.
While Figures \ref{fig:ef1d2iga} and  \ref{fig:ef1d3iga} show several sample plots of approximate and exact eigenfunctions, we show in this subsection convergence rates of the eigenfunction errors.

As the analysis predicts, the error in the eigenfunctions for the optimally-blended schemes does not exhibit extra orders of superconvergence. For the $H^1$-seminorm, all schemes yield a convergence of order $p$ for $p$-th order isogeometric elements with maximum continuity $C^{p-1}$ at element interfaces.

\begin{table}[ht]
\centering 
\begin{tabular}{| c | c | c | c | c |}
\hline
$N$ & $O_2$ & $G_{3}$ & $GL_{3}$  & $G_2$ \\[0.1cm] \hline
20 & 8.007409E-02 & 8.007620E-02 & 8.007355E-02 & 8.007966E-02 \\[0.1cm] 
40 & 1.962886E-02 & 1.962889E-02 & 1.962885E-02 & 1.962894E-02 \\[0.1cm] 
80 & 4.882983E-03 & 4.882983E-03 & 4.882983E-03 & 4.882984E-03 \\[0.1cm] 
160 & 1.219233E-03 & 1.219233E-03 & 1.219233E-03 & 1.219233E-03 \\[0.1cm] 
320 & 3.047138E-04 & 3.047138E-04 & 3.047138E-04 & 3.047138E-04 \\[0.1cm] 
Order& 2.01 & 2.01 & 2.01 & 2.01 \\ \hline
\end{tabular}
\caption{Eigenfunction $H^1$-seminorm errors for $u_3$ when using $C^1$ quadratic isogeometric analysis with different quadrature rules.}
\label{tab:efe2iga}
\end{table}

Figure \ref{fig:eferr1d2igah1} shows the $H^1$-seminorm errors of the eigenfunctions $u_3$ and $u_{11}$ for $C^1$  quadratic isogeometric analysis with quadrature rules $G_3, GL_3, G_2$, and $O_2$. The error convergence rate is two and the differences in the errors are fairly small. More precisely, Table \ref{tab:efe2iga} shows the errors for $u_3$. These numbers are the data for the left plot in Figure \ref{fig:eferr1d2igah1}. Their differences are in the order of $10^{-6}$ for the case with mesh size $1/20$ and of a scale of $10^{-7}$ for the case with mesh size $1/40$. As a consequence, they have the same convergence rates. Similar results are  observed for other eigenfunctions and details are omitted here.

For completeness, Figure \ref{fig:eferr1d2igal2} shows the eigenfunction errors in $L^2$-norm for $C^1$ quadratic case and Figure \ref{fig:eferr1d345iga} for higher order cases. Using Lemma \ref{lem:inv} and Theorem \ref{thm:efe}, we obtain an estimation for eigenfunction errors in $L^2$-norm
$
\| u_j - \tilde u_j^h \|_{0,\Omega} \le Ch^{p+1}
$ (see also \eqref{eq:efl2} in the \ref{sec:ee}).
All the schemes yield errors which are of order $p+1$ for the $p$-th order isogeometric analysis of continuity $C^{p-1}$. These errors have the optimal convergence orders. We observe a superconvergent rate of eigenfunction error in $L^2$-norm for quintic elements on coarse meshes. Again, the differences in errors when different quadrature rules are utilized are small.
These numerical results confirm the eigenfunction error analysis in Section \ref{sec:efana}.

\begin{figure}[ht]
\centering
\includegraphics[height=5.0cm]{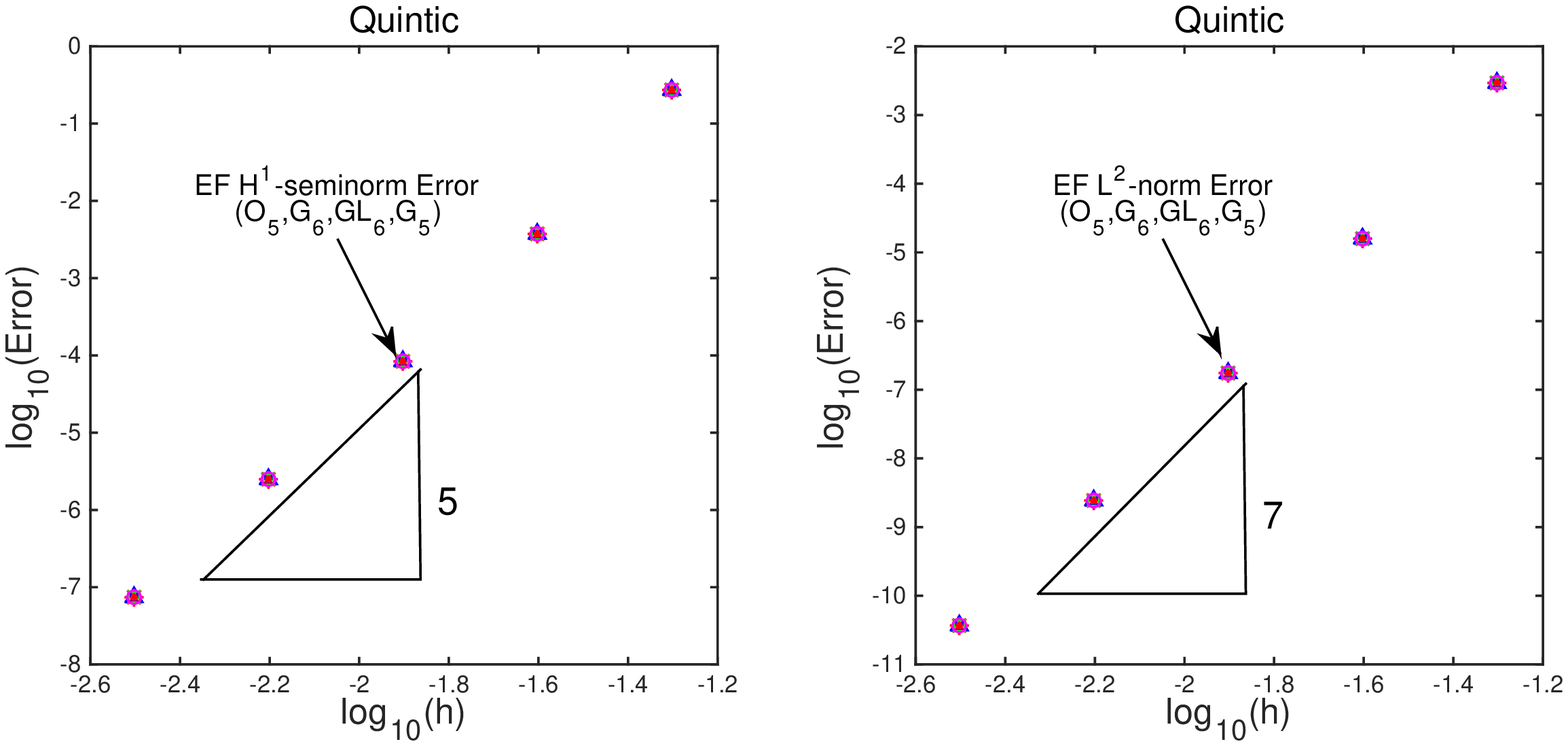}
\caption{Eigenfunction $H^1$-seminorm (left) and $L^2$-norm (right) errors for $u_{11}$ when using $C^4$ quintic  isogeometric analysis with quadrature rules $G_{6}, GL_{6}, G_5$, and $O_5$.}
\label{fig:eferr1d345iga}
\end{figure}

\subsection{Convergence study in two dimensions} The properties in two and higher dimensions are simple extensions of one dimensional cases when using tensor product meshes \cite{gao2013kronecker,gao2014fast}. Now, we consider the two-dimensional problem \eqref{eq:pdee} with eigenvalues and eigenfunctions 
$
\lambda_{jk} = \pi^2 (j^2 + k^2), u_{jk} = 2 \sin(j \pi x) \sin(k \pi y), j, k = 1, 2, \cdots.
$ 
Again, we assume that once we solve the eigenvalue problem, the numerical eigenvalues $\lambda_j^h$ and $\tilde \lambda_j^h$ are sorted in ascending order and paired with the true eigenvalues $\lambda_j$.
\begin{figure}[ht]
\centering
\includegraphics[height=5.0cm]{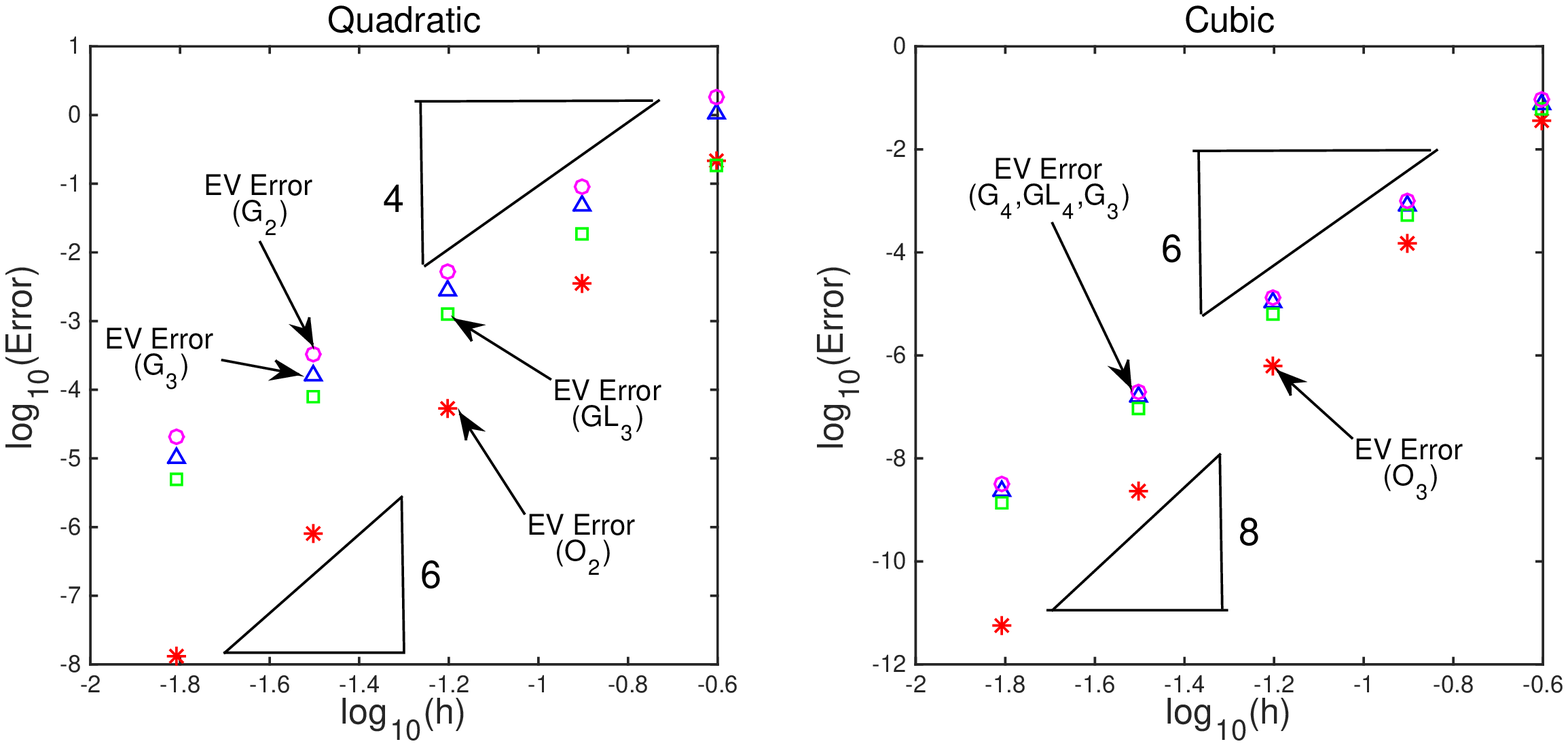} \\
\caption{Eigenvalue errors for $\lambda_{2,2}$ in 2D when using quadrature rules for $C^1$ quadratic and $C^2$ cubic isogeometric analysis.}
\label{fig:everr2d234iga}
\end{figure}

We focus on studying the convergence behavior when using the optimal blending rules. Figure \ref{fig:everr2d234iga} shows the eigenvalue errors for $\lambda_{2,2}$ when using quadrature rules for $C^1$ quadratic and $C^2$ cubic isogeometric elements. The mesh configurations for quadratic and cubic cases are $4\times 4, 8\times 8, 16\times 16, 32\times 32,$ and $64 \times 64$.
The eigenvalue errors of the optimal blending schemes are of order $2p+2$, while the standard quadratures yield order $2p$. These observations verify the analysis we discussed in Section \ref{sec:ea}.

\begin{table}[ht]
\centering 
\begin{tabular}{ | c | cc | cc | cc | cc |}
\hline
 & \multicolumn{2}{c|}{$p=1$} & \multicolumn{2}{c|}{$p=2$} & \multicolumn{2}{c|}{$p=3$} \\[0.1cm] 
$N$  & $\lambda_1$ &  $\lambda_4$ & $\lambda_1$ &  $\lambda_4$ & $\lambda_1$ &  $\lambda_4$  \\[0.1cm] \hline
 5&	1.00119&	0.88632&	0.96997&	0.67227&	0.92794&	0.64385\\[0.1cm]
10&	1.00040&	1.00127&	0.99231&	0.89264&	0.97974&	0.75699\\[0.1cm]
20&	1.00011&	1.00192&	0.99806&	0.97046&	0.99468&	0.92242\\[0.1cm]
40&	1.00003&	1.00057&	0.99952&	0.99242&	1.00223&	0.97899\\[0.1cm]
Order &1.83&	2.28&	1.96&	1.82&	1.7&	1.39 \\[0.1cm]\hline
\end{tabular}
\caption{Effectivity index (EI) of the error estimator $R(\tilde \lambda_j^h, \tilde u^h_j)$ and the convergence order of EI to 1.}
\label{tab:ee} 
\end{table}

\subsection{An error estimator}
We derive an error estimator in \ref{sec:ee} and study the numerical performance. The effectivity index of the error estimator $R(\tilde \lambda_j^h, \tilde u^h_j)$ is defined in \eqref{eq:ei}. Table \ref{tab:ee} shows the effectivity indices of the estimations of the first and the fourth eigenfunction errors in energy-norm when using isogeomtric elements with $p=1,2,3$ in one dimension. The domain $\overline\Omega=[0,1]$ is discretized uniformly with $ N = 5,10,20,40$ elements. Table \ref{tab:ee} shows that the effectivity indices are close to one and they converge to 1 at an almost quadratic rate, which confirms the discussion in \ref{sec:ee}.

\subsection{Computational efficiency} We now present the computational cost (time) while using the optimally-blended quadrature rules with a comparison to the case while using standard quadrature rules.

\begin{table}[ht]
\centering
\begin{tabular}{|c| c c| c c|}
\hline
 & \multicolumn{2}{c|}{Gauss rule} & \multicolumn{2}{c|}{Optimal rule} \\
 $p$ & Assembly time & Total time & Assembly time & Total time \\ \hline
2 & 1.10   & 41.54  & 1.24   & 41.18 \\
3 & 1.48   & 44.86  & 1.65   & 45.23 \\
4 & 1.71   & 47.12  & 1.91   & 47.25 \\
5 & 2.06   & 48.84  & 2.23   & 49.34 \\
6 & 2.36   & 50.93  & 2.51   & 51.07 \\
7 & 2.78   & 54.69  & 2.97   & 54.93 \\
 \hline
\end{tabular}
\caption{Wall-clock time (in seconds) required for the matrix assembly and the entire simulation of the eigenvalue problem. The test problem is \eqref{eq:pdee} with $\overline\Omega = [0,1]^2$, which is discretized uniformly with $100 \times 100$ elements.}
 \label{tab:cp}
\end{table}

Table \ref{tab:cp} shows the computational time of the matrix assembly using Gauss and optimal rules and the total simulation time for the eigenvalue problem \eqref{eq:pdee} in two dimensions. Herein, the matrices are assembled using the tensor product structure with the corresponding one-dimensional matrices. The total time is dominated by the solution of the generalized matrix eigenvalue problem \eqref{eq:mevp} which is the most time-consuming part of the simulation. As Table \ref{tab:cp} shows, the total time required to solve the eigenvalue problem is similar for both Gauss and optimal rules. The cost of the assembly of the tensor product matrix does not increase much even when we evaluate two quadrature rules instead of one when using the optimally-blended rules. Moreover, nonstandard quadrature rules (with a minimal number of quadrature points), which are equivalent to the optimally-blended quadrature rules in the sense of producing the same mass and stiffness matrices, are developed in \cite{deng2018dispersion} and they reduce computational cost; see also the reduced rules in \cite{hiemstra2017optimal}.

\section{Concluding remarks} \label{sec:conclusion}
In this paper, we  derive blending schemes to optimize the dispersion errors of isogeometric analysis. 
These optimally-blended quadrature rules lead to two extra order of convergence in the eigenvalue errors. Utilizing this superconvergence result, we derive an asymptotically-exact a posteriori error estimator. The optimal blending schemes are established for $p$ up to order seven  for maximum continuity spaces. 
We will seek to generalize these results in the future work to arbitrary polynomial and variable continuity orders. Another future work would be the generalization of the optimal blending technique to numerical methods for the differential eigenvalue problem with interfaces.

\section{Acknowledgments} 
This publication was made possible in part by the CSIRO Professorial Chair in Computational Geoscience at Curtin University and the Deep Earth Imaging Enterprise Future Science Platforms of the Commonwealth Scientific Industrial Research Organisation, CSIRO, of Australia. Additional support was provided by the European Union's Horizon 2020 Research and Innovation Program of the Marie Sk{\l}odowska-Curie grant agreement No. 777778, the Mega-grant of the Russian Federation Government (N 14.Y26.31.0013), the Curtin Institute for Computation, and Institute for Geoscience Research (TIGeR). The authors thank Eric Chung's constructive comments on the analysis.

\section*{References}


\bibliography{igaref}

\appendix{}
\section{An a posteriori error estimator} \label{sec:ee}

We revisit the generalized Pythagorean eigenvalue Theorem \ref{thm:gpet} to derive an a posteriori error estimator for the eigenfunctions. We present the result in the form of the following corollary.
\begin{coro} \label{coro:ee}
For each discrete mode, we assume the normalization $\| u_j \|_{0,\Omega} = 1$ and $\tilde b_h( \tilde u_j^h, \tilde u_j^h) = 1$. For sufficiently small $h$, we have the following error estimator with lower and upper bounds for the eigenfunction errors in the energy norm
\begin{equation} 
(1 - \rho) R(\tilde \lambda_j^h, \tilde u^h_j) \leq \| u_j - \tilde u_j^h \|_E \leq (1 + \rho) R(\tilde \lambda_j^h, \tilde u^h_j),
\end{equation}
where $\rho = \mathcal{O}(h^2), 0<\rho<1,$ and the residual is defined as
\begin{equation}
R(\tilde \lambda_j^h, \tilde u^h_j) = \sqrt{ \big| a(\tilde u_j^h, \tilde u_j^h) - \tilde\lambda_j^h b(\tilde u_j^h, \tilde u_j^h) \big| }.
\end{equation}
\end{coro}

\begin{proof} \label{rem:ee}

Firstly, the energy norm is the $H^1$-seminorm, thus we also have
\begin{equation} \label{eq:efen}
\| u_j - \tilde u_j^h \|_{E} \le Ch^{p}.
\end{equation}
Also from Lemma \ref{lem:inv},  we obtain an estimation for eigenfunction errors in $L^2$-norm
\begin{equation} \label{eq:efl2}
\| u_j - \tilde u_j^h \|_{0,\Omega} \le Ch^{p+1}.
\end{equation}

Now the Corollary can be established by applying the superconvergence property of the eigenvalue errors derived in subsection \ref{sec:evana} to the generalized Pythagorean eigenvalue theorem. Alternatively, we derive this as follows. 
\begin{equation*}
\begin{aligned}
\| u_j - \tilde u_j^h \|_E^2 & = a(u_j - \tilde u_j^h, u_j - \tilde u_j^h) \\
& = a(u_j, u_j ) - 2 \lambda_j b(u_j, \tilde u_j^h) + a(\tilde u_j^h, \tilde u_j^h) \\
& = \lambda_j b(u_j, u_j ) - 2 \lambda_j b(u_j, \tilde u_j^h) + \lambda_j b(\tilde u_j^h, \tilde u_j^h) -  \lambda_j b(\tilde u_j^h, \tilde u_j^h) + a(\tilde u_j^h, \tilde u_j^h) \\
& = \lambda_j \| u_j - \tilde u_j^h \|^2_{0, \Omega} - \lambda_j b(\tilde u_j^h, \tilde u_j^h) + a(\tilde u_j^h, \tilde u_j^h) - \tilde\lambda_j^h b(\tilde u_j^h, \tilde u_j^h) + \tilde\lambda_j^h b(\tilde u_j^h, \tilde u_j^h) \\
& = \lambda_j \| u_j - \tilde u_j^h \|^2_{0, \Omega} + ( \tilde\lambda_j^h  - \lambda_j) \| \tilde u_j^h \|^2_{0, \Omega} + \tilde R(\tilde \lambda_j^h, \tilde u^h_j),
\end{aligned}
\end{equation*}
where $$\tilde R(\tilde \lambda_j^h, \tilde u^h_j) = a(\tilde u_j^h, \tilde u_j^h) - \tilde\lambda_j^h b(\tilde u_j^h, \tilde u_j^h).$$ 
Now, using Theorem \ref{thm:eve} and \eqref{eq:efl2}, we have 
\begin{equation}
\lambda_j \| u_j - \tilde u_j^h \|^2_{0, \Omega} + ( \tilde\lambda_j^h  - \lambda_j) \| \tilde u_j^h \|^2_{0, \Omega}  \leq C h^{2p+2},
\end{equation}
which is of higher order than $\| u_j - \tilde u_j^h \|_E^2$ (which is of order $h^{2p}$), that is,
\begin{equation} \label{eq:eer}
\| u_j - \tilde u_j^h \|_E^2 = \mathcal{O}(h^{2p+2}) + \tilde R(\tilde \lambda_j^h, \tilde u^h_j).
\end{equation}
Clearly, for sufficiently small $h$, $\tilde R(\tilde \lambda_j^h, \tilde u^h_j) \ge 0.$ This allows us to write 
$$
\tilde R(\tilde \lambda_j^h, \tilde u^h_j) = R^2(\tilde \lambda_j^h, \tilde u^h_j).
$$
The error is dominated by the residual $R(\tilde \lambda_j^h, \tilde u^h_j)$ which is of order $h^{2p}$ in the view of \eqref{eq:al}, i.e., $R(\tilde \lambda_j^h, \tilde u^h_j) / h^{2p} = \mathcal{O}(1)$. For sufficiently small $h$, one can rewrite \eqref{eq:eer} as 
\begin{equation} 
-C h^{2p+2} + R^2(\tilde \lambda_j^h, \tilde u^h_j) \le \| u_j - \tilde u_j^h \|_E^2 \le C h^{2p+2} + R^2(\tilde \lambda_j^h, \tilde u^h_j),
\end{equation}
where $C$ is a positive constant independent of $h$. This further leads to
\begin{equation} \label{eq:417}
(1 - C h^2/ C_0 ) R^2(\tilde \lambda_j^h, \tilde u^h_j) \le \| u_j - \tilde u_j^h \|_E^2 \le (1 - C h^2/ C_0 ) R^2(\tilde \lambda_j^h, \tilde u^h_j),
\end{equation}
for an $h$-independent positive constant $C_0 = R^2(\tilde \lambda_j^h, \tilde u^h_j) / h^{2p}$.
Hence, taking square root of \eqref{eq:417} and using Taylor expansion to arrive to $\rho \approx C h^2/ (2C_0) = \mathcal{O}(h^2)$ yield the desired result.
\end{proof}

\begin{remark} \label{rem:apee}
This a posteriori error estimator strongly relies on the superconvergence result on uniform meshes. This estimator shares the feature of the asymptotically exact estimators. As a consequence of Corollary \ref{coro:ee}, we expect that the effectivity index, which is defined as
\begin{equation} \label{eq:ei}
EI = \frac{R(\tilde \lambda_j^h, \tilde u^h_j)}{\| u_j - \tilde u_j^h \|_E},
\end{equation}
converges to 1 at rate $\rho = \mathcal{O}(h^2)$. 
 We refer to \cite{cances2017guaranteed} for more reliable and robust a posteriori  error estimators. 
\end{remark}

\end{document}